\documentclass[english, 12pt]{article}
\usepackage{amsfonts}
\usepackage{amsmath}
\usepackage{amssymb}
\usepackage{amsthm}
\usepackage{amscd}
\usepackage{amscd}
\usepackage{amsthm}

\newtheorem{theo}{Theorem}[section]
\newtheorem{prop}{Proposition}[section]
\newtheorem{lem}{Lemma}[section]

\newtheorem{rk}{Remark}[section]

\newtheorem{exa}{Example}[section]
\newtheorem{coro}{Corollary}[section]
\numberwithin{equation}{section}

\DeclareMathOperator{\ran}{ran}
\DeclareMathOperator{\dom}{dom}

\newcommand{\aux}{\mathrm{aux}}

\newcommand{\Dir}{\mathrm{D}}

\newcommand{\ccalE}{\check{\cal{E}}}

\newcommand{\ctl}{\check{T}_t(\lam)}
\newcommand{\calD}{{\cal{D}}}
\newcommand{\calH}{{\cal{H}}}
\newcommand{\calHaux}{{\cal{H}}_{\rm aux}}
\newcommand{\Har}{{\cal{H}}_{\rm har}^J}

\def\R{{\mathbb{R}}}
\def\N{\mathbb{N}}
\def\C{\mathbb{C}}

\def\calE{{\cal{E}}}

\newcommand{\cel}{\check{\calE}_\lambda}

\newcommand{\lam}{\lambda}

\newcommand{\Om}{\Omega}

\newcommand{\from}{\colon}

\begin{document}
\bibliographystyle{alpha}

\title{ Spectral asymptotic and positivity for singular Dirichlet-to-Neumann operators }

\author{\normalsize   Ali BenAmor\footnote{University of Sousse. High School for Transport and Logistics.
 4023 Cit\'e Erriadh, Sousse. Tunisia. E-mail: ali.benamor@ipeit.rnu.tn}
}
\date{}
\maketitle
\begin{abstract}
In the framework of Hilbert spaces we shall give necessary and sufficient conditions to define a Dirichlet-to-Neumann operator via Dirichlet principle. Analyzing singular Dirichlet-to-Neumann operators, we will establish Laurent expansion near singularities as well as  Mittag--Leffler expansion for the related quadratic forms. The established  results will be exploited to  solve definitively the problem of  positivity of the related semigroup in the framework of Lebesgue spaces. The obtained results are supported by some examples where we analyze properties of singular Dirichlet-to-Neumann operators related to  Neumann and Robin Laplacian on Lipschitz domains. Among other results,  we shall demonstrate that regularity of the boundary may affect positivity and derive Mittag-Leffler expansion for the eigenvalues of singular Dirichlet-to-Neumann operators.
\end{abstract}

\section{Introduction}

Singular Dirichlet-to-Neumann (D-to-N for short) operators are one parameter family of operators which have standing singularities at eigenvalues of some Dirichlet operator. As illustration, consider the following example. Let $\calE$ be the  quadratic form associated with Neumann's Laplacian on the unit disc $\mathbb{D}$ in $\R^2$:
\begin{eqnarray*}
D(\calE)= H^1(\mathbb{D}),\ \calE[u] = \int_{\mathbb{D}} |\nabla u|^2\,dx,\ \forall\,u\in H^1(\mathbb{D} ).
\end{eqnarray*}
Let $dS$ be the surface measure on $\Gamma$. Set $L^2(\Gamma):=L^2(\Gamma,dS)$ and let $J$ be the operator 'trace on the boundary'
\[
J:H^1(\mathbb{D})\to L^2(\Gamma),\ u\mapsto u|_\Gamma.
\]
It is well known that  $J$ is bounded and
\[
\ran J=H^{1/2}(\Gamma),\ \ker J=H_0^1(\mathbb{D}).
\]
Let $-\Delta_{\mathbb{D}}$ be the Dirichlet Laplacian on $\mathbb{D}$. It is nothing else but the selfadjoint operator associated to the quadratic form $\calE|_{\ker J}$.\\
Let $\lam\in\R$ such that $\lam$ is not an eigenvalue of $-\Delta_{\mathbb{D}}$, $\psi\in H^{1/2}(\Gamma)$ and  $u_\lam\in H^1(\mathbb{D})$ be the unique solution of the boundary value problem
\begin{eqnarray*}
\label{trace1}
\left\{\begin{gathered}
-\Delta u_\lam -\lambda u_\lam=0, \quad \hbox{in } \mathbb{D}\\
u_\lam= \psi,~~~{\rm on}\  \Gamma
\end{gathered}
\right..
\label{Laplace}
\end{eqnarray*}
Define the form $\cel$ as follows:
\begin{eqnarray*}
D(\cel)= H^{1/2}(\Gamma),\ \cel[\psi]:=\calE_\lam[u_\lam]=\int_\mathbb{D} |\nabla u_\lam|^2\,dx -\lam\int_\mathbb{D} u_\lam^2\,dx\
\forall\,\psi\in H^{1/2}(\Gamma).
\end{eqnarray*}
Then $\cel$ is a closed lower semibounded and densely defined quadratic form in $L^2(\Gamma)$ (see \cite{Arendt-Mazzeo}). Let us denote by $\check L_\lam$ the selfadjoint operator associated to $\cel$ via Kato's representation theorem. The family $\check L_\lam$ is a typical family of singular Dirichlet-to-Neumann operators with singularities being  the Dirichlet eigenvalues. In fact, in case  $\lam$ is a Dirichlet eigenvalue then the above mentioned boundary value problem is not uniquely solvable: it has as many linearly independent solutions  as the multiplicity of $\lam$. Hence $\cel$ is not well defined and one  can not define a quadratic form by the described procedure. This explains the connotation 'singular'.\\
The Dirichlet eigenvalues are also singularities for the quadratic form $\cel$, the resolvent, the semigroup and the eigenvalues of  $\check L_\lam$.\\
Let us observe that, according to Dirichlet principle, it holds
\begin{align*}
\cel[\psi]=\inf\{\int_{\mathbb{D}} |\nabla u|^2\,dx -\lam\int_{\mathbb{D}} u^2\,dx,\ u\in H^1(\mathbb{D}),\  Ju=\psi\}.
\end{align*}
Now some natural questions arise: How do spectral objects of $\check L_\lam$,  behave near these singularities? How do these singularities influence positivity or sub-Markov property for the related semigroup?\\
In this paper we shall first, put the above procedure in the abstract setting of Hilbert spaces in order to   construct D-to-N operators in Hilbert spaces. Concretely, let $\calH, \calHaux$ be Hilbert spaces, $\calE$ a lower semibounded quadratic form with domain $\calD\subset \calH$ and $J$ a linear operator $J:\dom J\subset\calD\to \calHaux$ with dense range. We  shall give necessary and sufficient conditions ensuring Dirichlet principle to hold. Precisely we aim for  finding  necessary and sufficient conditions so that
\begin{align}
\ccalE[Ju]&:=\inf\{\calE(v,v)\colon\ v\in\mathcal{L}(u)\}\nonumber\\
&=\min\{\calE(v,v)\colon\ v\in\mathcal{L}(u) \},\ \forall\,u\in\calD,
\label{DP}
\end{align}
with unique minimizer, where $\mathcal{L}(u)$ is a linear manifold (to be determined) depending on the vector $u$.\\
It turns out that (\ref{DP}) holds if and only if $\dom J$ is the direct sum of the kernel of $J$ and some specific subspace. Then we shall give necessary and sufficient conditions for the quadratic form $\ccalE$ defined via Dirichlet principle, to be lower semibounded and closed in $\calHaux$. The obtained form is commonly named the trace form of $\calE$ and the related operator is commonly named the Dirichlet-to-Neumann operator.\\
Once the construction has been done we consider a positive form $\calE$ and $J$ such that $\ker J$ is dense in $\calHaux$ and the form $\calE_\Dir:=\calE|_{\ker J}$ is closed and  has discrete spectrum.  We construct the singular form $\cel$, trace of the form $\calE-\lam$, where $\lam\in\R$ is not an eigenvalue of $\calE_\Dir$. Our major contribution in this respect is to write a representation formula for $\cel$  (Theorem \ref{decomposition1}). The formula involves $\ccalE_0$, some Dirichlet operator and an abstract Poisson kernel operator. It plays a central role in the development of the paper.\\
Then we shall turn our attention to analyze some properties of the singular  D-to-N operator $\check{L}_\lam$ associated with $\cel$. Extending $\cel$ to complex values $z$, we shall show that $\ccalE_z$ is meromorphic with simple poles coinciding with the eigenvalues of  $\calE_\Dir$. At this stage our main contribution is to establish Laurent and Mittag--Leffler expansions for $\ccalE_z$ (Theorem \ref{AbstractLaurent} and Theorem \ref{ML}).\\
As a byproduct, we shall determine the exact rate of growth for $|\ccalE_z[\psi]|$ as $z$ approaches any singularity. The main input for proving the mentioned results is the representation formula of $\ccalE_z$.\\
In case $\check{L}_\lam$ has compact resolvent, pushing our analysis forward, we shall examine the behavior of the eigenvalues near the singularities.\\
As applications we shall consider the special case of $L^2$ spaces. Our major contribution in this framework is to utilize  the obtained asymptotic from  former sections to establish necessary and sufficient conditions ensuring positivity preservation property of the obtained semigroup near singularities (see Theorem \ref{Positivity-Multiple}). These conditions involve some abstract Poisson kernel operator and the eigenfunctions of the singularities. Thereby we completely solve the problem of positivity preservation in a general framework. Here the main ingredient is the Mittag--Leffler expansion for the trace form together with Beurling--Deny criterion.\\
Finally, we analyze the singular D-to-N operators related to Neumann and Robin  Laplacian on  Euclidean bounded Lipschitz domains. Here we shall be able to write Mittag--Leffler expansion of $\cel$ with coefficients depending only on Dirichlet eigenvalues and  boundary integrals of the normal derivatives of the related eigenfunctions. Moreover, in some cases we shall write Mittag--Leffler expansion for the eigenvalues of the singular D-to-N operator. The expansion involves the Dirichlet eigenvalues solely!\\
Besides, we shall demonstrate that regularity of the boundary may affect positivity as well as multiplicities of Dirichlet eigenvalues. Whereas for negative $\lam$ the semigroup is even ultracontractive.\\
We quote that construction of D-to-N operators in the setting of Hilbert spaces via Dirichlet principle was already performed by many authors \cite{Arendt1,Post,BBST}, following different approaches and under more restrictive  assumptions, such as $j$-ellipticity, elliptic regularity or closedness of $\calE$.\\
Whereas, some spectral properties for the D-to-N operator  on Lipschitz domains were established in \cite{Arendt-Mazzeo,Behrndt-Elst} and  analysis of positivity preservation for intervals and the unit disc was elaborated in \cite{Daners}.\\
However, as long as we know, there is no systematic studies neither concerning spectral asymptotic near singularities nor concerning positivity property of the semigroup related to singular D-to-N operators.
\section{D-to-N operators via Dirichlet principle}
Let $\calH, \calHaux$ be two Hilbert spaces.
Let $(\cdot,\cdot)$ and $(\cdot,\cdot)_{\aux}$ denote the scalar products on $\calH$ and $\calHaux$, respectively and $\|\cdot\|, \|\cdot\|_\aux$ be the corresponding norms.\\
We shall use the connotation 'form' for any sesquilinear symmetric form as for the related quadratic form.\\
Let $\calE$ be a lower semibounded form with domain $\calD\subseteq \calH$.
For $u\in\calD$ we abbreviate $\calE[u]:=\calE(u,u)$ and for every $\lambda\in\R$ we set
\[
\calE_\lam,\ \dom\calE_\lam=\calD,\  \calE_\lam[u] := \calE[u] - \lambda\|u\|^2.
\]
Assume we are given a linear operator $J\from \dom J\subseteq\calD\to \calHaux$ with dense range.

\subsection{The positive case}
Assume that $\calE$ is positive.\\
Being inspired by \cite[Theorem 2.1, Lemma 2.8]{BBST} let us define $\ccalE$ with domain in $\calHaux$  as follows:
\begin{equation}
  \dom \ccalE= \ran J,\quad \ccalE[Ju]:= \inf\{{\calE}[v]:\,v\in\dom J,\  Jv=Ju\}.
  \label{construction1}
\end{equation}
\begin{lem}
The functional $\ccalE$ is a quadratic form, i.e. it comes from a sesquilinear form.
\label{quadratic-1}
\end{lem}
\begin{proof}
For every $\lam<0$ we define
\begin{equation}
\ccalE_\lam:\  \dom \ccalE_\lam = \ran J,\quad \ccalE_\lam[Ju]:= \inf\{\calE_\lam[v]:\,v\in\dom J,\  Jv=Ju\}.
\end{equation}
Plainly $\calE_\lam^{1/2}$ is a norm on $\calD$. Let $\calD_\lam$ be the completion of $\calD$ with respect to $\calE_\lam^{1/2}$. We still denote by $\calE_\lam$ the extension of $\calE_\lam$ on $\calD_\lam$.\\
Let $u\in\dom J$. Then there is a sequence $(v_n)\subset\dom J$ such that $Jv_n=Ju$ for every $n$ and $\lim_{n\to\infty}\calE_\lam[v_n]=\ccalE_\lam[Ju]$. The fact that $J(1/2 v_n + 1/2 v_m)=Ju$ in conjunction with the identity
\[
\calE_\lam[\frac{1}{2}(v_n - v_m)] + \calE_\lam[\frac{1}{2}(v_n + v_m)]
= \frac{1}{2}\calE_\lam[v_n] + \frac{1}{2}\calE_\lam[v_m],
\]
imply that the sequence $(v_n)$ is in fact $\calE_\lam$-Cauchy. Hence there is $P_\lam u\in\calD_\lam$ such that $\ccalE_\lam[Ju]= \calE_\lam[P_\lam u]$. Moreover as $\calE_\lam^{1/2}$ is a norm we conclude that $P_\lambda u$ is unique.\\
On the other hand, as $J$ is linear, $\calE_\lam$ is a quadratic form and $P_\lam u$ is unique we obtain  that $\ccalE_\lam$ is a quadratic form for any $\lam<0$.\\
From the very definition we infer that the family $\ccalE_\lam$ is monotone decreasing as $\lam\uparrow 0$. Now the proof follows the lines of the proof of \cite[Lemma 2.8]{BBST}, which we reproduce for the convenience of the reader.\\
We claim that
\[
\ccalE[Ju]=\lim_{\lam\uparrow 0}\ccalE_\lam[Ju],\ \forall\,u\in\dom J.
\]
Indeed, the monotony of $\ccalE_\lam$ implies that $\lim_{\lam\uparrow 0}\ccalE_\lam[Ju]=\inf_{\lam<0}\ccalE_\lam[Ju]$. Accordingly, we obtain $\lim_{\lam\uparrow 0}\ccalE_\lam[Ju]\geq\ccalE[Ju]$.\\
Conversely, let $v\in\dom J$ be such that $Jv=Ju$. Then $\ccalE_\lam[Ju]\leq\calE_\lam[v]$. Thus $\lim_{\lam\uparrow 0}\ccalE_\lam[Ju]\leq \calE[v]$. Passing to the infimum on $v$ we obtain $\lim_{\lam\uparrow 0}\ccalE_\lam[Ju]\leq \ccalE[Ju]$ and the claim is proved.
\end{proof}

We aim to find necessary and sufficient conditions for the quadratic form $\ccalE$ to obeys the Dirichlet principle. Namely, conditions ensuring
\[
\inf\{{\calE}[v]:\,v\in\dom J,\  Jv=Ju\}=\min \{{\calE}[v]:\,v\in\dom J,\  Jv=Ju\},
\]
with a unique minimizer.\\
To achieve our goal  we introduce the linear subspace defined by
\[
\Har:=\{ u\in\dom J,\ \calE(u,v)=0,\ \forall\,v\in\ker J\},
\]
and for every $u\in\dom J$ we designate by $C_u$ the linear manifold
\[
C_u:=u+\ker J.
\]
The subscript 'har' stands for 'harmonic' as indicated by the first example.\\
Let us observe that  the infimum can be written as
\begin{align*}
  \inf\{{\calE}[v]:\; v\in C_u\}.
\end{align*}
Let us first solve the uniqueness problem.
\begin{theo}
Assume that for each $u\in\dom J$ the infimum is attained at some $Pu$. Then
\begin{enumerate}
\item  $Pu$ should satisfies
\begin{equation}
Pu\in C_u,\ \calE(Pu,v)=0,\ \forall\,v\in\ker J.
\end{equation}
In particular, $Pu\in\Har$ and
\[
 \ccalE[Ju]= \min\{\calE[v],\ v\in\Har\cap C_u\}.
\]
\item {\em Uniqueness.} $Pu$ is unique if and only if
\begin{equation}
\Har\cap \ker J=\{0\}.
\label{Unique}
\end{equation}
\end{enumerate}
\label{CharacUnique}
\end{theo}
\begin{proof}
Let $u\in\dom J$. Assume that the infimum is attained at some $Pu$. From the definition $Pu\in C_u$.\\
Now let $v\in\ker J$. Then for any  $t>0$, $v_t:=tv+Pu\in C_u$. Hence $\calE[v_t]\geq \calE[Pu]$. An elementary computation leads to
\begin{align*}
\calE[v_t]=t^2\calE[v] + 2t{\rm Re}\,\calE(v,Pu) + \calE[Pu]\geq \calE[Pu].
\end{align*}
Dividing by $t$ and letting $t\downarrow 0$ yields ${\rm Re}\,\calE(v, Pu)\geq 0$. Changing $v$ by $-v$ leads to ${\rm Re}\,\calE(v, Pu)\geq 0$.\\
Similarly, changing $v_t$ by $w_t:=itv+Pu$ we obtain ${\rm Im}\,\calE(v, Pu)= 0$. Thus $\calE(v, Pu)\geq 0$ and  $Pu\in\Har$.\\
As $Pu\in\Har\cap C_u\subset C_u$ we achieve
\[
\calE[Pu]\leq    \inf\{\calE[v],\ v\in\Har\cap C_u\}\leq \calE[Pu],
\]
which ends the proof of the first assertion.\\
Uniqueness:  Assume that $Pu$ is unique for every $u\in\dom J$. In particular for $u\in\ker J$ we get $\ccalE[Ju]=0$ and by uniqueness $Pu=0$. Now let $v\in\Har\cap\ker J$. Then $\calE[v]=0$. By uniqueness, once again, we obtain $v=Pu=0$. Thereby  $\Har\cap\ker J=\{0\}$.\\
Conversely assume that $\Har\cap \ker J=\{0\}$ and that for some $u\in\dom J$ the infimum is attained at $Pu,P'u$. Then $Pu-P'u\in\Har\cap \ker J=\{0\}$ and then $Pu=P'u$.
\end{proof}
\begin{rk}
{\rm
\begin{enumerate}

\item The $Pu$ might be non-unique. It is for instance the case if $0$ is an eigenvalue of $\calE$.
\item Under assumption of Theorem \ref{CharacUnique} the map
\[
\dom J\to \Har,\ u\mapsto P u,
\]
is linear.
\end{enumerate}
}
\end{rk}
From now on we maintain the assumption (\ref{Unique}).\\
The converse of the latter theorem solves the problem $\inf=\min$.
\begin{theo}
Assume that
\begin{equation}
\Har\cap C_u\neq\emptyset,\ \forall\,u\in\dom J.
\label{Min}
\end{equation}
Let $Pu$ be any element from  $\Har\cap C_u$. Then
\begin{equation}
\inf\{\calE[v],\ v\in C_u\}=\calE[Pu].
\end{equation}
\label{InfMin}
\end{theo}
\begin{proof}
For each $v\in\dom J$, we set $Pv$ any element from $\Har\cap C_v$. From the very definition we get
\[
\inf\{\calE[w],\ w\in C_u \}\leq \calE[Pu],\ \forall\,u\in\dom J.
\]
On the other hand we have $v-Pv\in\ker J$. Since in particular $Pv\in\Har$ we get $\calE(Pv,Pv-v)=0$ and then
\[
\calE[Pv]=\calE(Pv,v),\ \forall\,v\in\dom J.
\]
On the other hand the positivity of $\calE$ together with the latter identity lead to
\begin{align}
0\leq \calE[Pv-v]=\calE[Pv]-2\calE(Pv,v)+\calE[v]=-\calE[Pv]+\calE[v].
\end{align}
Hence we achieve
\begin{align}
\calE[v]\geq \calE[Pv],\ \forall\,v\in\dom J.
\end{align}

Let now $u\in\dom J$ and $v\in C_u$. Then, as $Pv\in C_v$ we get $JPv=Jv=Ju$. Thereby $Pv\in C_u$ and $Pu-Pv\in\ker J\cap\Har$. Hence $\calE[Pu-Pv]=0$. Set $w=Pu-Pv$. Since $w\in\ker J$ and $Pv\in\Har$ we obtain $\calE(w,Pv)=0$.  A straightforward computation leads to
\begin{align}
\calE[Pu]=\calE[w+Pv]=\calE[w] + 2\calE(w,Pv) + \calE[Pv]=\calE[Pv].
\end{align}
Finally putting all together we achieve
\[
\inf\{\calE[v],\ v\in C_u\}\geq \inf\{\calE[Pv],\ v\in C_u\}=\calE[Pu],
\]
and hence
\[
\inf\{\calE[v],\ v\in C_u\}=\calE[Pu].
\]
which completes the proof.

\end{proof}
\begin{rk}
{\rm
One may think that condition (\ref{Min}) is equivalent to $\Har\cap\ker J\neq\emptyset$. However, it is not true. Indeed, let $E_0$ be the smallest eigenvalue of the Dirichlet Laplacian on the unit ball.  Let
\[
\dom\calE=H^1(B),\ \calE[v]=\int_B|\nabla v|^2\,dx-E_0\int_B v^2\,dx.
\]
Then $\Har\cap\ker J$ is the linear span of the eigenfunction of $E_0$, say $u_0$. Let $u\in H^1(B)\setminus H^1_0(B)$. Then any $v\in\Har$ should be proportional to $u_0$. On the other hand any $v\in C_u$ can not be in $H^1_0(B)$, which is a contradiction. Hence  $\Har\cap\ C_u =\emptyset$ whereas $\Har\cap\ker J\neq\emptyset$.
}
\label{quadratic}
\end{rk}

Combining both theorems we obtain:
\begin{coro}
The infimum is attained and is unique for every $u\in\dom J$ if and only if $\Har\cap C_u$ is a singleton.
\end{coro}
The latter theorems together with the corollary have very far-reaching consequences.
\begin{theo}
\begin{enumerate}
\item The infimum is attained and is unique for every $u\in\dom J$, if and only if
\begin{equation}
\dom J=\Har\oplus\ker J.
\label{directSum}
\end{equation}
\item Assume that the decomposition (\ref{directSum}) holds. Let $u\in\dom J$ and  $P u\in\Har$  be the component of $u$ corresponding to the direct sum. Then
\begin{equation}
\inf\{{\calE}[v]:\; v\in C_u\}=\calE[P u].
\end{equation}
\end{enumerate}
\label{FinalInf}
\end{theo}

\begin{proof}
Assume that for every $u\in\dom J$ the infimum is attained at some $Pu$ and is unique. Then according to Theorem \ref{FinalInf}, $Pu\in\Har\cap C_u$. Moreover $u$  can be written, in a unique manner, as $u=Pu + u-Pu$ with $Pu\in\Har$ and $u-Pu\in\ker J$.\\
The converse: If (\ref{directSum}) is fulfilled,  mimicking the proof of Theorem \ref{InfMin}, one shows that
\[
\inf\{\calE[v],\ v\in C_u\}=\calE[P u],
\]
and the infimum is attained at the sole element $P u$.
\end{proof}
The following gives sufficient conditions for the decomposition $\dom J=\ker J\oplus \Har$ to hold.
\begin{prop}
Assume that $J$ is injective. Then the Dirichlet principle holds. Moreover,
\[
\ccalE[Ju]=\calE[u],\ \forall\,u\in\dom J.
\]
\end{prop}
\begin{proof}
By assumption we obtain $\ker J=\{0\}$ and $\dom J=\Har$. Besides if $u,v\in\dom J$ are such that $J v=J u$, then $u=v$ and hence $\ccalE[Ju]=\calE[u]$.
\end{proof}
Once we have solved the problem of existence and uniqueness of the minimizer, we are in a comfortable situation to discuss the closedness $\ccalE$.\\
We quote that under condition of Theorem \ref{FinalInf} we have
\[
\dom\ccalE=\ran J,\ \ccalE[Ju] = \calE[Pu],\ \forall\,u\in\dom J.
\]
Here is an improvement of \cite[Lemma 3.4]{BBST}.\\
We set $\calE^J$ the form defined by
\[
\dom\calE^J=\dom J,\ \calE^J[u] :=\calE[u] + \|Ju\|^2_\aux.
\]
Observe that if the direct sum decomposition (\ref{directSum}) is fulfilled then $\ccalE^J$ defines a scalar product on $\Har$.
\begin{theo}
Suppose that condition (\ref{directSum}) is fulfilled. Then $\ccalE$ is closed if and only if $(\Har,\calE^J)$ is a Hilbert space.
\label{closed-Positive}
\end{theo}
\begin{proof}
Assume that $\ccalE$ is closed. Let $(u_n)$ be a $\calE^J$-Cauchy sequence in $\Har$. Owing to the fact that $Pu_n=u_n$ for every integer $n$, we get
\[
\ccalE[Ju_n-Ju_m]=\calE[u_n-u_m]\to 0,\ \text{and}\ \|Ju_n-Ju_m\|_\aux\to 0.
\]
Thereby $(Ju_n)$ is $\ccalE_{-1}$-Cauchy and by closedness of $\ccalE$ there is $u\in\dom J$ such that
\[
\lim_{n\to\infty}\ccalE[Ju_n -Ju] + \|Ju_n -Ju\|_\aux=0.
\]
Recalling that $\ccalE[Ju_n-Ju]=\calE[u_n-Pu]$ and $Ju=JPu$ and that $Pu\in\Har$ we conclude that $(u_n)$ is $\calE^J$-convergent to $Pu$ and hence converges  in $(\Har,\calE^J)$.\\
Conversely, assume that $(\Har,\calE^J)$ is a Hilbert space. let $(Ju_n)$ and $v\in\calHaux$ be such that $Ju_n\to v$ and $(Ju_n)$ is $\ccalE_{-1}$-Cauchy. Then $(Pu_n)$ is a Cauchy sequence in $(\Har,\calE^J)$. Thereby there is $u\in\Har$ such that $\calE^J[Pu_n - u]= \calE^J[Pu_n - Pu]\to 0$. Thus $v=Ju$ and $\ccalE[Ju_n-Ju]=\calE[Pu_n-Pu]\to 0$, showing hat $\ccalE$ is closed.
\end{proof}
Here is an example which illustrates the case where $J$ is injective.
\begin{exa}
{\rm
Let $\calH=L^2(\R):=L^2(\R,dx)$ and let $\calE$ be defined by
\[
\calD=\{u\in L^2_{\rm loc}(\R),\ u'\in L^2(\R)\},\ \calE[u] = \int_\R (u')^2\,dx.
\]
Let $\mu$ be a positive Radon measure on $\R$ such that $\mu(I)>0$ for every nonempty open interval of $\R$. We recall that every element from the space $\calD$ has a continuous representative with respect to Lebesgue measure. Thus without loss of generality we shall assume that elements from $\calD$ has been chosen to be continuous.\\
We choose $\calHaux= L^2(\R,\mu)$ and  consider
\[
J:\dom J=\calD\cap L^2(\R,\mu)\to L^2(\R,\mu),\ u\mapsto u.
\]
Thus $J$ is well defined, moreover it is injective. Indeed, if $u=0\,m-a.e.$. Then by continuity of $u$ the set $I=\{|u|>0\}$ is open and has zero $m$-measure. By assumptions on the measure $m$ we conclude that $I$ is empty and hence $u$ vanishes identically.\\
According to our previous results the trace form $\ccalE$ is given by
\[
\dom\ccalE=\calD,\  \ccalE[Ju]=\ccalE[u]=\calE[u],\ \forall\,u\in\calD.
\]
Consequently, $\ccalE$ is nothing else but the restriction of $\calE$ to $\dom J$.\\
We claim that $\ccalE$ is closed. To that end we make use of Theorem \ref{closed-Positive}.\\
Let $(u_n)\subset\dom J$ be such that $\calE[u_n - u_m] + \int_\R(u_n - u_m)^2\,d\mu\to 0$. Then there is $u\in L^2(\R,\mu)$ and $v\in L^2(\R)$ such that
\[
\int_\R(u_n - u)^2\,d\mu\to 0\ \text{and}\ \int_\R(u_n' - v)^2\,dx\to 0.
\]
From the elementary identity
\[
u_n(x) - u_n(y) =\int_y^x u_n'\,dt,
\]
we derive
\[
(u_n(x) - u_n(y))^2 \leq |x-y|\calE[u_n].
\]
Let $[a,b]\subset\R$ be compact. Then the latter inequality leads to
\[
u_n(x)^2\leq  2|x-y|\calE[u_n] + 2u_n^2(y),\ \forall\,x,y\in[a,b].
\]
Integrating with respect to $\mu$ we obtain

\[
\sup_{x\in[a,b]}u_n(x)^2\leq  2\frac{\max(b-a,1)}{\mu([b-a])}\big(\calE[u_n] + \int_\R u_n^2\,d\mu\big).
\]
Thus the sequence $(u_n)$ converges locally uniformly to a continuous function on $\R$.  Thereby we can assume that $u$ is continuous and $u_n\to u$ locally uniformly on $\R$. In particular, $(u_n')$ converges to $u'$ in the sense of distributions. Thus $v=u'$. Finally we obtain $u\in\calD\cap L^2(\R,\mu)$ and $\calE[u_n-u] + \int_\R (u_n - u)^2\,d\mu\to 0$. Hence $\ccalE$ is closed.\\
Let us quote that the selafdjoint operator related to $\ccalE$ is commonly noted by $-\frac{d}{d\mu}\frac{d}{dx}$. In particular if $\mu$ is Lebesgue measure then the operator is simply the Laplacian on $\R$.

}
\end{exa}

\subsection{The general case}
Here we no longer assume positivity of $\calE$, but lower semi-boundedness: there is a real $c\geq 0$ such that
\[
\calE[u]\geq -c\|u\|^2,\ \forall\,u\in\dom\calE.
\]
One is tempted to define $\ccalE$ via formula (\ref{construction1}). However, one is faced to a new problem with lower semi-boundedness of $\ccalE$. If $\ccalE$ were defined by (\ref{construction1}) and if it were lower semibounded then there should be a real constant $c'\geq 0$ such that
\[
\ccalE[Ju]\geq -c'\|Ju\|\aux^2,\ \forall\,u\in\dom J.
\]
This leads, in particular to
\[
\calE[v]\geq 0,\ \forall\,v\in\ker J.
\]
Unfortunately this is a strong restriction.\\
Being inspired by Theorem \ref{CharacUnique}, we define
\begin{equation}
\ccalE[Ju]:= \inf\{\calE[v],\ v\in\Har\cap C_u\},\ \forall\,u\in\dom J.
\end{equation}
We stress that, thanks to Theorem 2.1, this definition coincides with the former one for positive forms.\\
Let us first solve the problem of lower semi-boundedness of $\ccalE$.
\begin{lem}
The functional  $\ccalE$ is lower semibounded if and only if there is a real $c'\geq 0$ such that
\begin{align}
\calE[v]\geq -c'\|Jv\|_\aux^2,\ \forall\,v\in\Har.
\label{lsb}
\end{align}
\end{lem}
\begin{proof}
Obviously if $\ccalE$ is lower semibounded then inequality (\ref{lsb}) holds true.\\
Conversely, if (\ref{lsb}) holds true, then for any $u\in\dom J,\ v\in\Har\cap C_u$ we obtain
\[
\calE[v]\geq -c'\|Jv\|_\aux^2=-c'\|Ju\|_\aux^2.
\]
Taking the infimum over $\Har\cap C_u$ we deduce that $\ccalE$ is lower semibounded.
\end{proof}
\begin{lem}
Assume that the lower bound (\ref{lsb}) holds. Then $\ccalE$ is a quadratic form.
\end{lem}
\begin{proof}
Let $\alpha\geq c'$. Set $Q:=\calE +\alpha J$. Then $Q$ is positive and hence according to Lemma \ref{quadratic-1} the trace of $Q$ with respect to $J$, which we denote by $\check{Q}$, is a quadratic form. On the other hand from the definition we infer that
\[
\check{Q}[Ju]=\ccalE[Ju] + \alpha\|Ju\|_\aux^2,\ \forall\,u\in\dom J.
\]
Thus $\ccalE$ is a quadratic from as well.

\end{proof}

Concerning uniqueness of infimum, Theorem 2.1 still holds  in this general framework. Under condition (\ref{lsb}), Theorem 2.2 still holds true as well. For, one has to change $\calE$ by $\calE +c' J$. Moreover, Theorem 2.4. still holds in this general framework.
\begin{theo}
\begin{enumerate}
\item Assume that (\ref{directSum}) holds. Then  For every $u\in\dom J$ there is a unique $Pu\in\Har\cap C_u$ such that $\ccalE[Ju]=\calE[Pu]$.
\item Assume that (\ref{directSum}) and (\ref{lsb}) hold. Then the form $\ccalE$ is closed if and only if $(\Har,\calE^{(1+c')J})$ is a Hilbert space.
\end{enumerate}
\label{D-to-N-general}
\end{theo}
The proof runs exactly as the positive case, so we omit it.
\begin{rk}
{\rm
Following different methods,  under assumption (\ref{directSum}), it was constructed in \cite[Theorem 2.5]{Arendt1} and in \cite[Proposition 3.3]{BBST} a closed quadratic form in $\calHaux$. In fact all obtained forms coincide with $\ccalE$.
}
\end{rk}
\section{The singular D-to-N operator}
Let $\calE$ be a positive form and  $\lam\in\R$. Set
\[
\calE_\lam:=\calE-\lam.
\]
We introduce the quadratic form
\[
\calE_\Dir:\dom\calE_\Dir=\ker J,\ \calE_\Dir[u]=\calE[u].
\]
The subscript $\Dir$ stands for 'Dirichlet'.\\
Let us stress that the positivity assumption for $\calE$ is not crucial. For, if $\calE$ is lower semibounded one can shift it to get a positive form.\\
We assume that $\ker J$ is dense in $\calH$ and  $\calE_\Dir$ is closed. Let $L_\Dir$ be the positive selfadjoint operator related to $\calE_\Dir$. We suppose that $L_\Dir$ has compact resolvent and designate by $\sigma_e(L_\Dir)$ the set of eigenvalues of $L_\Dir$.\\
In order do construct $\cel$ via the Dirichlet principle, we should have, among other conditions
\[
\Har(\lambda)\oplus\ker J=\dom J,
\]
where $\Har(\lam)=\{u\in\dom J,\ \calE_\lam(u,v)=0,\ \forall\,v\in\ker J\}$.\\
The condition  $\Har(\lambda)\cap\ker J=\{0\}$ forces $\lambda$ not to be an eigenvalue of $L_D$. Hence from know on we assume
\begin{align}
\lambda\in\R\setminus \sigma_e(L_\Dir).
\label{NotEigen}
\end{align}
\begin{lem} The following two conditions are equivalent:
\begin{enumerate}
\item $\dom J=\Har\oplus\ker J$.
\item $0$ is not an eigenvalue of $L_\Dir$ and $\dom J=\Har(\lambda)\oplus\ker J$, for every  $\lambda\in(\R\setminus\sigma_e(L_\Dir))$.
\end{enumerate}
\label{EquivSum}
\end{lem}
\begin{proof}
The implication $2\Rightarrow 1$ is obvious. Let us prove the reversed implication.\\
Suppose $\Har\oplus\ker J=\dom J$. Plainly we obtain $0$ is not an eigenvalue of $L_\Dir$.\\
Let $u\in\dom J$. We already know from Theorem \ref{CharacUnique} that if $\Har\cap\ker J=\{0\}$ then $\Har\cap C_u=\{Pu\}$. Let $\lambda\in\R\setminus\sigma_e(L_\Dir)$. Set
\[
w_\lambda:= \lambda(L_\Dir -\lambda)^{-1}Pu,\ u_\lam:= Pu + w_\lambda.
\]
Then $w_\lambda\in\ker J$. On the one hand $Ju_\lambda =JP u= Ju$, yielding thereby $u-u_\lambda\in\ker J$. On the other one, a straightforward computation shows that  $u_\lambda \in \Har(\lambda)$. Hence  $u=(u-u_\lambda) + u_\lambda$ is the sum of an element from $\ker J$ and an element from $\Har(\lambda)$. Let us prove uniqueness of the latter decomposition. Let $u\in \Har(\lambda)\cap\ker J$. Then $\calE(u,v)=\lam(u,v)$ for any $v\in\ker J$. As $u\in\ker J$,  if $u\neq 0$ then $\lam$ is an eigenvalue of $L_\Dir$, which is a contradiction and the proof is finished.
\end{proof}
%
As we aim for defining $\cel$ via Dirichlet principle, we adopt from now on,  the following assumption:
\begin{equation}
\dom J=\Har\oplus\ker J.
\label{lambda-direct}
\end{equation}
We mention that assumption (\ref{lambda-direct}) implies that $0$ is not an eigenvalue for $\calE_\Dir$.\\
Under assumption (\ref{lambda-direct}), according to the latter lemma together with Theorem \ref{D-to-N-general}, we are able to define $\cel$ via the Dirichlet principle:
\begin{align*}
\dom\cel=\ran J,\ \cel[Ju]&=\inf\{\calE_\lam[v],\ v\in\Har(\lam)\cap C_u\}\\
&=\min\{\calE_\lam[v],\ v\in\Har(\lam)\cap C_u\} =\calE_\lambda[P_\lambda u],
\end{align*}
for any $\lam \in\R\setminus\sigma_e(L_\Dir)$. Here $P_\lam u$ is the component of $u$ from $\Har(\lam)$ corresponding to the direct sum decomposition $\dom J=\Har(\lam)\oplus\ker J$.\\
For $\lam=0$ we shall denote $\ccalE_0$ simply by $\ccalE$ and $P_0$ by $P$.\\
Let us define the abstract {\em Poisson kernel operator}, $\Pi$ as follows:
\begin{equation}
\Pi:\dom\Pi=\ran J\subset\calH\to \Har,\ \text{such that}\ \Pi J=P.
\end{equation}
Observe that assumption (\ref{lambda-direct}) ensures that $\Pi$ is well defined.
\begin{lem}
The operator $\Pi$ is an isometric isomorphism from the normed space $(\ran J,\check\calE_{-1})$ into the normed  space $(\Har,\calE^J)$. Moreover, it holds
\[
\Pi=(J|_{\Har})^{-1}.
\]
\label{Poisson}
\end{lem}
\begin{proof}
If $\Pi Ju_1=\Pi Ju_2$, then $Pu_1=Pu_2$. Hence $JPu_1=Ju_1=JPu_2=Ju_2$ and  $\Pi$ is injective. The direct sum decomposition yields surjectivity of $\Pi$.\\
Let $\psi=Ju$. A straightforward computation leads to
\begin{equation}
\calE^J[\Pi\psi]= \calE[Pu] +\|JPu\|_\aux^2=\check\calE[\psi] + \|\psi\|_\aux^2.
\end{equation}
Hence $\Pi$ is an isometry.\\
Finally, from the definition of $\Pi$ we infer
\[
\Pi Ju= u,\ \forall\,u\in\Har,
\]
which yields the second claim.
\end{proof}
\begin{rk}
{\rm
At this stage we would like to emphasize that in contrast to the theory elaborated in \cite{Post} we do neither assume that $\calE$ is closed nor that it is densely defined nor that $\dom J=\calD$ and $J:(\calD,\calE_{-1}^{1/2})\to\calHaux$ is bounded. However, if it is the case and if $0$ is not an eigenvalue of $L_\Dir$, then an obvious modification of the proof of \cite[Proposition 2.15]{Post} in conjunction with Lemma \ref{EquivSum} lead to the decomposition (\ref{lambda-direct}).\\
Furthermore if we set $\Pi_{-1}$ the Poisson kernel operator related to $\calE_{-1}$ then $\Pi_{-1}$ coincide with what is called in \cite{Post} the 'Dirichlet solution map'.

}
\end{rk}

We are in position now to establish a representation formula for $\cel$ which will play a decisive role for investigating its properties.
\begin{theo}[A representation formula]
Let $u\in\dom J$ and $\psi=Ju$. Then
\begin{equation}
\check\calE_\lambda[\psi]= \check{\calE}[\psi] - \lam \big(L_\Dir(L_\Dir-\lam )^{-1}\Pi\psi,\Pi\psi\big).
\label{ForDec}
\end{equation}
\label{decomposition1}
\end{theo}

\begin{proof}

Let $u,\psi$ be as in the theorem. Set
\[
K:=L_\Dir^{-1},\ u_\lam:=P_\lambda u,\ v:= \Pi\psi\ \text{and}\ w_\lambda:=\lam(1-\lam K)^{-1}Kv = \lam K(1-\lam K)^{-1}v.
\]
Then $v\in\Har$, $Jv=\psi=Ju$ and $w_\lam\in\ker J$.
Let $w\in\ker J$. Then
\begin{align}
\calE_\lam(v+w_\lam,w)&=\calE_\lam(v,w) + \calE_\lam(w_\lam,w)= -\lam(v,w) + \lam\calE_{\Dir,\lam}((L_\Dir -\lam)^{-1}v,w)\nonumber\\
&=-\lam(v,w)+\lam(v,w)=0.
\end{align}
Hence $u_\lam = v + w_\lam$.\\
By definition of $\cel$ we have  $\check\calE_\lam[\psi]=\check\calE_\lam[u_\lam]$. Observing that $v$ and $w_\lam$ are $\calE$-orthogonal, we accordingly  obtain:
\begin{eqnarray}
\check\calE_\lam[\psi] &=& \check\calE_\lam[u_\lam]= \calE[v+w_\lam] -\lam\|v+w_\lam\|^2=\calE[v] + \calE[w_\lam] - \lam\|v+w_\lam\|^2 \nonumber\\
&=&\ccalE[\psi] +\calE_\lambda[w_\lambda] -\lam\|v\|^2 - 2\lam{\rm Re}\,(v,w_\lambda).
\end{eqnarray}
As $w_\lam=\lam(L_\Dir -\lam)^{-1}v$ we get
\[
\calE_\lam[w_\lam]= \lam(v,w_\lam).
\]
In particular,  $(v,w_\lam)$ is real. Thus
\[
\check{\calE}_\lambda[\psi]=\check{\calE}[\psi] -\lam\| v\|^2 -\lambda( v,w_\lambda).
\]
Having the formulae of $v$ and $w_\lam$ in mind we achieve
\begin{eqnarray}
 \cel[\psi] &=& \check{\calE}[\psi] -\lam\|\Pi\psi\|^2 -\lam^2((1-\lam K)^{-1}K\Pi\psi,\Pi\psi)\nonumber\\
 &=& \check{\calE}[\psi] - \lam(L_\Dir(L_\Dir-\lam)^{-1}\Pi\psi,\Pi\psi),
 \label{decomp}
\end{eqnarray}
and the proof is finished.

\end{proof}
\begin{rk}
{\rm
Formula (\ref{ForDec}) highlights the connection between, Dirichlet Laplacian, Dirichlet principle, Poisson kernel, and D-to-N operator. Furthermore it highlights the singular part of $\cel$.

}
\end{rk}
For $\cel$ to define a lower semibounded closed form for any $\lam\in\R\setminus\sigma(L_\Dir)$, according to Theorem \ref{D-to-N-general}, we have to impose further restrictions. Hence from know on we shall assume, unless otherwise stated, that: for any $\lam \in\R\setminus\sigma_e(L_\Dir) $ there is $c_\lam> 0$ such that
\begin{equation}
\calE_{\lam}[u]\geq -c_\lam\|J u\|_\aux^2,\ \forall\,u\in\Har(\lam)
\label{lambda-lsb}
\end{equation}
and
\begin{equation}
(\Har(\lam),\calE_\lam^{(1+c_\lam)J})\ \text{is a Hilbert space}.
\label{lambda-closed}
\end{equation}
On the light of Theorem \ref{D-to-N-general}, under assumptions  (\ref{lambda-lsb})-(\ref{lambda-closed}) together with Lemma \ref{EquivSum}, the form $\cel$ is lower semibounded densely defined  and closed.\\
Henceforth, we designate by $\check{L}_\lam$ the selfadjoint operator related to $\cel$ via Kato representation theorem. For $\lam=0$, the operator $\check{L}_0$ will be denoted simply by  $\check{L}$.\\
At his stage we would like to emphasize that similar construction for $\cel$ was developed in \cite{Post} via the concept of 'boundary pairs'. However, under the additional stronger assumptions that $\calE$ is closed, $J:(\calD,\calE_{-1}^{1/2})\to\calHaux$ is bounded and the boundary pair is elliptically regular.

\section{The asymptotic}
In order to perform asymptotic in the complex plane we shall first extend the trace from to complex numbers by extending formula (\ref{ForDec}). Precisely, for every  $z\in\C\setminus\sigma_e(L_\Dir)$ we define
\begin{equation}
\check\calE_z[\psi]:= \check{\calE}[\psi] - z \big(L_\Dir(L_\Dir-z )^{-1}\Pi\psi,\Pi\psi\big),\ \forall\,\psi\in\ran J.
\label{ComplexDecomp}
\end{equation}
Formula (\ref{ComplexDecomp}) shows that the mapping
\[
z\mapsto \check\calE_z[\psi],
\]
is meromorphic with poles the eigenvalues of $L_\Dir$. Owing to selfadjointness of $L_\Dir$ they are all simple poles.\\
From now on we designate by  $E$ any eigenvalue of $L_\Dir$ and $P_E$ its associated eigenprojection.
\subsection{Laurent and Mittag--Leffler expansions for the form}
\begin{theo}[Laurent expansion]
Let $E$ be an eigenvalue of $L_\Dir$ and $P_E$ be its associated eigenprojection. Let $C_E$ be a positively oriented small circle around $E$. Set
\[
A_0=\frac{1}{2i\pi}\int_{C_E}(z-E)^{-1} (L_\Dir -z)^{-1}\,dz\ \text{and}\ r_E= \|A_0\|.
\]

Then for every $\psi\in\ran J$ it holds
\begin{align}
\check\calE_z[\psi]&= \ccalE[\psi] + \frac{z}{z -E}\|L_\Dir^{\frac{1}{2}}P_E\Pi\psi\|^2 -z \sum_{k=0}^\infty (z-E)^k \|L_\Dir^{\frac{1}{2}} A_0^{\frac{k+1}{2}}\Pi\psi\|^2,\ 0<|z - E|<r_E,
\label{LaurentS1}
\end{align}
where the series is absolutely convergent.
\label{AbstractLaurent}
\end{theo}
\begin{proof}
From the standard theory of meromorphic operator valued  functions and since $E$ is a simple pole for $(L_\Dir-z )^{-1}$, the following Laurent expansion holds true
\[
(L_\Dir-z )^{-1} = \frac{A_{-1}}{z-E} + \sum_{k=0}^{\infty} (z - E)^k A_k,\ 0<|z - E|<r_E\ \text{uniformly},
\]
where $A_k=A_0^{k+1}$ and $A_{-1}=-P_E$. Finally making use of the representation formula (\ref{ComplexDecomp}) for $\check{\calE}_z$ together with continuity of the scalar product, we get the desired expansion.
\end{proof}
\begin{coro}
The following asymptotic behavior is true:
\[
\lim_{z\to E}\big((z-E)\ccalE_z[\psi]\big)=  E\|L_\Dir^{1/2}P_E\Pi\psi\|^2,\  \forall\,\psi\in\ran J.
\]
It follows in particular,
\begin{enumerate}
\item $\lim_{\lam\uparrow E}\cel[\psi]=-\infty$ for all $\psi\in\ran J$.
\item $\lim_{\lam\downarrow E}\cel[\psi]=\infty$ for all $\psi\in\ran J$.
\item $|\ccalE_z[\psi]|$ grows as fast as $E|z - E|^{-1}\|L_\Dir^{1/2}P_E\Pi\psi\|^2$ when approaching the singularity $E$.
\end{enumerate}
\end{coro}
The corollary derives directly from Theorem \ref{AbstractLaurent}, so we omit its proof.\\
We proceed now to establish Mittag-Leffler expansion for $\ccalE_z$.\\
Let $E_0\leq E_1\leq\cdots\leq E_k\cdots$ be the increasing arrangement for the eigenvalues of $L_\Dir$ where each $E_k$ is repeated as many times as its multiplicity. Let $(u_k)$ be the corresponding orthonormal basis of eigenfunctions.
\begin{theo}[Mittag-Leffler expansion]
Let $z\in\C\setminus\sigma_e(L_\Dir)$ and  $\psi\in\ran J$. Then
\begin{equation}
\ccalE_z[\psi] = \ccalE[\psi] + z\sum_{k=0}^\infty \frac{E_k}{z-E_k}\left| (\Pi\psi,u_k) \right|^2,
\label{Mittag}
\end{equation}
where the series converges absolutely.
\label{ML}
\end{theo}
\begin{proof}
By the spectral theorem we get, for every $u\in \calH$
\[
L_\Dir(L_\Dir -z)^{-1}u= \sum_{k=0}^\infty \frac{E_k}{E_k - z} (u,u_k)u_k.
\]
Making use of the representation formula (\ref{ComplexDecomp}) together with the continuity of the scalar product we obtain the sought formula.
\end{proof}
\begin{rk}
{\rm
\begin{enumerate}
\item The connotation  'Mittag--Leffler expansion' is justified by the fact that the expansion can be written in the form
\[
\ccalE_z[\psi] = \ccalE[\psi] + z \sum_{k=0}^\infty \left| (\Pi\psi,u_{k}) \right|^2
+ \sum_{k=0}^\infty \big( \frac{E_k^2}{z-E_k} + E_k (z+E_k)\big) \left| (\Pi\psi,u_{k}) \right|^2.
\]
This is plainly the Mittag-Leffler expansion for $\ccalE_z$.
\item For later use, we emphasize that the expansion can also  be written in an other form. For, let $m_k$ be the multiplicity of $E_k$ and  $(u_{1k},\cdots,u_{m_{k}k})$ be an orthonormal eigenbasis for $E_k$. Then according to the expansion (\ref{Mittag}) we have

\begin{equation}
\ccalE_z[\psi] = \ccalE[\psi] + z\sum_{k=0}^\infty \frac{E_k}{z-E_k}\sum_{l=1}^{m_k} \left| (\Pi\psi,u_{lk}) \right|^2.
\label{Mittag-2}
\end{equation}
\end{enumerate}
}
\end{rk}

\subsection{The eigenvalues near the poles}

Assume that $\check{L}_\lam$ has compact resolvent for some (and hence every) $\lam\in\R\setminus\sigma_e(L_\Dir)$. Let us turn our attention to study properties of eigenvalues of $\check L_\lambda$ near the singularities. To that and we shall establish a monotony property for $\cel$.
\begin{lem}
For every fixed $\psi\in\ran J$, the mapping $\lam\mapsto\cel[\psi]$ is strictly decreasing on each interval of $\R\setminus\sigma_e(L_\Dir)$.
\label{decrease}
\end{lem}

\begin{proof}
According to formula (\ref{decomposition1}), the form $\cel[\psi]$ is $\lam$-differentiable. Moreover,
making use of the first resolvent formula, we obtain
\begin{equation}
\frac{d}{d\lam}\cel[\psi]= - (L_\Dir(L_\Dir - \lambda)^{-2} \Pi\psi,\Pi\psi) =-\| L_\Dir^{1/2}(L_\Dir -\lam)^{-1}\Pi\psi\|^2\leq 0,\ \forall\,\lambda\in\R\setminus\sigma_e(L),
\end{equation}
which was to be proved.
\end{proof}

\begin{theo}
Let $\check E(\lambda)$ be an eigenvalue of $\cel$.
\begin{enumerate}
\item The mapping
\[
\R\setminus\sigma_e(L_\Dir)\to\R,\  \lambda\mapsto\check{E}(\lambda),
\]
is strictly decreasing on each interval of $\R\setminus\sigma(L_\Dir)$.
\item Let $E$ be any eigenvalue of $L_\Dir$ which is a singularity for $\check{E}(\lam)$. Then
\[
\lim_{\lambda\uparrow E}\check{E}(\lambda)=-\infty,\ \lim_{\lambda\downarrow E}\check{E}(\lambda)=\infty.
\]
\end{enumerate}
\end{theo}
\begin{proof}
The first assertion  is consequence of the min-max principle for successive eigenvalues together with Lemma \ref{decrease}. The second assertion follows from monotony of $\check{E}(\lam)$ and the fact that $E$ is a singularity for $\check{E}(\lam)$.
\end{proof}
\begin{rk}
{\rm
We stress that not every eigenvalue of $L_\Dir$ is a singularity for $\check{E}(\lam)$. Concrete examples for this fact can be found in \cite{Daners} or in the examples analyzed at the end of the current paper.
}
\end{rk}
\section{Positivity preservation}
In this section we assume that $\calH=L^2(X,m)$ and $\calHaux=L^2(X,\mu)$ (real Hilbert spaces), where $(X,m), (X,\mu)$ are $\sigma$-finite measure spaces and $m,\mu$ are positive measures on some $\sigma$-algebras of $X$.\\
We maintain the assumption that $\calE$ is positive together with assumptions (\ref{lambda-direct})-(\ref{lambda-lsb})-(\ref{lambda-closed}) from the latter section.\\
Furthermore, we  assume that the form $\calE$ is  closed, densely defined and is a  semi-Dirichlet form, i.e. its related semigroups
\begin{equation}
e^{-tL},\ t>0\ \text{is positivity preserving}.
\end{equation}
Equivalently,
\[
u\in\dom\calE\Rightarrow\,|u|\in\dom\calE\ \text{and}\ \calE[|u|]\leq \calE[u],
\]
or (Beurling--Deny criterion)
\[
u\in\dom\calE\Rightarrow\,u^\pm\in\dom\calE\ \text{and}\ \calE(u^+,u^-)\leq 0.
\]
Thereby the form   $\calE_{\Dir}$ is a  semi-Dirichlet form as well and hence its related semigroup, $e^{-tL_\Dir},\ t>0$ is also positivity preserving.\\
Obviously $\calE_{\lam}$ is a semi-Dirichlet form for every $\lam\in\R$.\\
Let $\ctl:=e^{-t\check{L}_\lam},\ t>0$ be the semigroup related to the form $\cel$. We shall exploit the already establish asymptotic to  discuss into which extend the positivity preservation property is inherited by the semigroup of the D-to-N operator, $\ctl$. It is expected that positivity property will depend on $\lambda$.\\
We shall use the abbreviation p.p. to mean 'positivity preserving'.\\
At this stage we mention that  some partial results concerning positivity in one and two dimensions can be found in \cite{Daners} and for bounded Lipschitz domains in \cite{Arendt-Mazzeo}.\\
It is  not possible to go ahead without some additional assumptions on the map $J$. Henceforth we assume that $\dom J=\dom\calE=\calD$ furthermore
\begin{equation}
u\in\calD\Rightarrow\, |Ju|=J|u|.
\label{J-positivity}
\end{equation}
Let us first investigate positivity of $\Pi$.
\begin{lem}
Let $\psi\in\ran J$ be positive. Then $\Pi\psi$ is positive as well.
\label{Pi-positivity}
\end{lem}
\begin{proof}
Let $\psi\in\ran J$ be positive and $u\in\calD$ such that $Ju=\psi$. By assumption (\ref{J-positivity}) we get $Ju=J|u|\geq 0$. Thus we may and shall assume that $u\geq 0$. By assumption (\ref{J-positivity}) once again we obtain $J(|Pu|)=|J Pu|=Ju$. Owing to Dirichlet principle we get $\ccalE[Ju]=\calE[Pu]\leq \calE[|Pu|]$. On the other hand as $\calE$ is a semi-Dirichlet form we get $\calE[|Pu|]\leq\calE[Pu]$. Hence from uniqueness we derive $Pu=|Pu|\geq 0$. Thus $\Pi\psi=Pu\geq 0$. \end{proof}

\begin{lem}
$\check{T}_t(0)$ is p.p.
\label{O-p.p.}
\end{lem}
\begin{proof}
We shall prove that $\ccalE$ is a semi-Dirichlet form. Let $u\in\calD$. Then by Dirichlet principle we have
\[
\ccalE[|Ju|]=\ccalE[J|u|]=\inf\{\calE[v],\ v\in\calD,\,Jv=J|u|\}.
\]
Now if $v$ is such that $Jv=J|u|$ then, by assumption (\ref{J-positivity}) we get $J|v|=|Jv|=J|u|=Jv$. Hence making use of the semi-Dirichlet property for $\calE$ we achieve
\begin{align*}
\ccalE[|Ju|]&\leq\inf\{\calE[|v|],\ v\in\calD\,Jv=Ju\}\\
&\leq \inf\{\calE[v],\ v\in\calD\,Jv=Ju\}=\ccalE[Ju],
\end{align*}
which completes the proof.
\end{proof}

\begin{prop}
Let $E_0$ be the smallest eigenvalue of $L_\Dir$. Then for every $\lam<E_0$, the semigroup $\ctl$ is p.p.
\label{Positivity-Under}
\end{prop}
\begin{proof}
{\em Step 1:} $\lam\leq 0$. Then $\calE_\lam$ is a positive semi-Dirichlet form. Hence Lemma \ref{O-p.p.} applied to $\cel$ instead of $\ccalE$ yields the p.p. for $\ctl$.\\
{\em Step 2:} $0<\lam<E_0$. Here we use Beurling--Deny criterion together with the representation formula. By polarization we get
\begin{align}
\cel(\psi^+,\psi^-)&=\check{\calE}(\psi^+,\psi^-) -\lam\big(L_\Dir(L_\Dir -\lam)^{-1}\Pi\psi^+,\Pi\psi^-\big)\nonumber\\
&= \check{\calE}(\psi^+,\psi^-) -\lam(\Pi\psi^+,\Pi\psi^-) -\lam^2\big((L_\Dir -\lam)^{-1}\Pi\psi^+,\Pi\psi^-\big)
\end{align}
According to Lemma \ref{O-p.p.} the first term is negative, whereas  the second term is negative owing to Lemma \ref{Pi-positivity}. As $\lam<E_0$,  we have
\[
(L_\Dir -\lam)^{-1} = \int_0^\infty e^{-\lam t}T_tu\,dt.
\]
As $T_t$ is p.p. we conclude that $(L_\Dir -\lam)^{-1}$ is p.p. as well. Hence the third term is also negative, leading to $\cel(\psi^+,\psi^-)\leq 0$ and the proof is finished.

\end{proof}
We proceed now to analyze positivity of $\ctl$ for $\lam>E_0$. To achieve our purpose we shall utilize Mittag--Leffler expansion for $\cel$.
\begin{theo}
Let $\lam>E_0$ and $E$ be an eigenvalue of $L_\Dir$ with multiplicity $m$. Let $(v_1,\cdots,v_m)$ be an orthonormal basis for $\ker(L_\Dir-E)$. Then $\ctl,\ t>0$ is p.p. on a left (resp. right) neighborhood  of $E$ if and only if one of the following equivalent conditions is fulfilled
\begin{enumerate}
\item
\[
\sum_{k=1}^m(\Pi\psi^+,v_k)\cdot(\Pi\psi^-,v_k)\geq 0,\ (\text{resp.} \leq 0),\ \forall\,\psi\in\ran J.
\]
\item
\[
(L_\Dir P_E \Pi\psi^+,\Pi\psi^-) \geq 0,\ (\text{resp.} \leq 0),\ \forall\,\psi\in\ran J.
\]
\end{enumerate}
\label{Positivity-Multiple}
\end{theo}
\begin{proof}
We shall use Beurling--Deny criterion. Let $\psi\in\ran J$. By polarization and  according to formula (\ref{ML}) we obtain
\begin{equation}
\ccalE_\lam(\psi^+,\psi^-) = \ccalE(\psi^+,\psi^-) + \lam\sum_{k=0}^\infty \frac{E_k}{\lam-E_k} (\Pi\psi^+,u_k) (\Pi\psi^-,u_k).
\end{equation}
Thereby the leading term in the expansion near $E$ is
\[
 \frac{\lam E}{\lam-E} \sum_{k=1}^m (\Pi\psi^+,v_k)\cdot(\Pi\psi^-,v_k).
\]
Accordingly, $\ccalE(\psi^+,\psi^-)$ is negative in a left (resp. right) neighborhood of $E$ if and only if
\[
\sum_{k=1}^m (\Pi\psi^+,v_k)\cdot(\Pi\psi^-,v_k)
\geq 0,\ (\text{resp.}\,\leq 0).
\]
Now the selfadjointness of $L_\Dir$ yields $\ran P_E=\ker (L_\Dir -E)$. Hence $P_E=\sum_{k=1}^m(\cdot,v_k)v_k$ and
\[
(L_\Dir P_E\Pi\psi^+,\Pi\psi^-)=E\sum_{k=1}^m(\Pi\psi^+,v_k)(\Pi\psi^-, v_k).
\]
Thus both conditions of the theorem are equivalent  and by Beurling--Deny criterion they are both equivalent to positivity to the left (resp. to the right).

\end{proof}

\begin{rk}
{\rm
The latter theorem yields the following observations for the semigroup $\ctl$ with $\lam>E_0$:
\begin{enumerate}
\item The semigroup $\ctl$ can not be simultaneously p.p. on both sides of any singularity.
\item If it is p.p. on one side then it is necessary non p.p. on the other one.
\item It might be non p.p. on both sides of some singularities.
\end{enumerate}
}
\end{rk}
Theorem \ref{Positivity-Multiple} leads immediately to the following conclusions:
\begin{coro}
Under assumptions of Theorem \ref{Positivity-Multiple}, suppose there are $\psi_1,\psi_2\in\ran J$ such that the sums
\begin{equation*}
\sum_{k=1}^m (\Pi\psi_1^+,u_k)\cdot(\Pi\psi_1^-,u_k)\ \text{and}\ \sum_{k=1}^m (\Pi\psi_2^+,u_k)\cdot(\Pi\psi_2^-,u_k),
\end{equation*}
have opposite signs. Then $\ctl$ is not p.p. in any neighborhood  of $E$.
\label{Left-Right}
\end{coro}
\begin{coro}
Assume that $E$ is simple and that the associated eigenfunction,  $u_E$ can be chosen to have constant sign. Then $\ctl$ is p.p. to the left of $E$ whereas it is non p.p. to the right of $E$.
\end{coro}
By the end of this section we shall demonstrate how to use information on p.p. for $\ctl$ to decide on p.p. for the semigroup related to the trace of another form.\\
Let $Q$ be an other positive densely defined semi-Dirichlet form such that
\[
\dom Q=\dom\calE\ \text{and}\ Q_\Dir=\calE_\Dir.
\]
Set
\[
\Har(Q):\{u\in\dom J,\ Q(u,v)= 0,\ \forall\,v\in\ker J\},
\]
and assume further that $\Har=\Har(Q)$.\\
Then $\calE_\Dir$ and $Q_\Dir$ have the same spectrum. Furthermore,  for any $\lam\in\R\setminus\sigma_e(L_\Dir)$ the trace form $\check{Q}_\lam$ is well defined via Dirichlet principle, is closed and densely defined in $\calHaux$.\\
We shall give in the last section a convincing example where all these conditions are fulfilled.\\
Denote by $\check{S_t}(\lam)$ the semigroup related to $\check{Q}_\lam$.
\begin{theo}
The semigroups $\ctl$ and $\check{S_t}(\lam)$ behave likely regarding positivity preservation property.
\label{SameBehav}
\end{theo}
\begin{proof}
As $\Har=\Har(Q)$ we obtain that both forms have the same Poisson kernel operator. Hence both trace forms have the same singular parts in their respective Laurent and Mittag-Leffler expansion. Now the  assertion of the theorem follows from Theorem \ref{Positivity-Multiple}.
\end{proof}

\section{Examples}
\subsection{The singular D-to-N operator on Lipschitz domains}

In this section we shall use the elaborated theory from former sections to study D-to-N operators related to Neumann Laplacian on Lipschitz domains.\\
Let $\Om\subset\R^d$ be a nonempty  open bounded connected subset with Lipschitz boundary $\Gamma$ and $\calE$
the gradient Dirichlet form on $H^1(\Om)$:
\begin{eqnarray}
D(\calE)= H^1(\Om),\ \calE[u] = \int_\Om |\nabla u|^2\,dx,\ \forall\,u\in H^1(\Om).
\end{eqnarray}
It is well known that the quadratic form $\calE$ is closed and densely defined  in $L^2:=L^2(\Om,dx)$. Moreover, $\calE$ is a Dirichlet form, i.e.,
\[
u\in H^1(\Om)\Rightarrow u_{0,1}:=(u\vee 0)\wedge 1\in H^1(\Om)\ \text{and}\ \calE[u_{0,1}]\leq \calE[u].
\]
The positive selfadjoint operator associated to  $\calE$, which we denote by $L$ is commonly named  the Neumann Laplacian on $\Om$. As $\calE$ is a Dirichlet form its related semigroup $e^{-tL},\ t>0$ is Markovian and hence is p.p.\\
Let $dS$ be the surface measure on $\Gamma$ (the $(d-1)$-Hausdorff measure of $\Gamma$). Set $L^2(\Gamma):=L^2(\Gamma,dS)$ and let $J$ be the operator 'trace to the boundary'
\[
J:H^1(\Om)\to L^2(\Gamma),\ J\mapsto u|_\Gamma.
\]
Then  $J$ is bounded (see Proposition \ref{Compact}). Moreover it is well known that
\[
\ker J=H_0^1(\Om)\ \text{and}\ \ran J=H^{1/2}(\Gamma),
\]
and $J$ has dense range. Thus $\calE_\Dir$ is the closed quadratic from associated with the Dirichlet Laplacian on $\Om$ and is a Dirichlet form as well. Owing to boundedness  of $\Om$ its well known that $L_\Dir$ has compact resolvent with simple smallest eigenvalue $E_0>0$.\\
In this case we have
\[
\Har(\lam)=\{u\in H^1(\Om),\  -\Delta u-\lam u=0\ \text{on}\ \Om\}.
\]
According to \cite[Theorem 10.1]{Fabes}, for any $\psi\in H^{1/2}(\Gamma)$ there is a unique $u\in H^1(\Om)$ such that

\begin{eqnarray*}
\left\{\begin{gathered}
-\Delta u =0, \quad \hbox{in } \Om\\
u= \psi,~~~{\rm on}\  \Gamma
\end{gathered}
\right..
\end{eqnarray*}
All these considerations lead to the decomposition $H^1(\Om)=H^1_0(\Om)\oplus\Har$.\\
Thus all conditions are fulfilled to define $\cel$ via Dirichlet principle. In fact, let $\lam\in\R\setminus\sigma_e(L_\Dir)$ and $\psi\in H^{1/2}(\Gamma)$. Then $P_\lam u$ is the unique element from  $H^1(\Om)$ which solves  the boundary value problem
\begin{eqnarray*}
\left\{\begin{gathered}
-\Delta P_\lam u -\lambda P_\lam u=0, \quad \hbox{in } \Om,\\
P_\lam u= \psi,~~~{\rm on}\  \Gamma
\end{gathered}
\right.
\end{eqnarray*}
and
\begin{eqnarray*}
D(\cel)= H^{1/2}(\Gamma),\ \cel[\psi] = \calE_\lam[P_\lam u]=\int_\Om |\nabla P_\lam u|^2\,dx-\lam\int_\Om (P_\lam u)^2\,dx\
,\ \forall\,\psi\in H^{1/2}(\Gamma).
\end{eqnarray*}
We proceed to show that $\cel$ is lower semibounded.
\begin{lem}
There is a finite constant $c>0$ such that
\begin{eqnarray}
\int_\Om (\Pi\psi)^2\,dx\leq c\int_\Gamma \psi^2\,d\Gamma,\ \forall\,\psi\in H^{1/2}(\Gamma).
\end{eqnarray}
\label{Pi-continuity}
\end{lem}
\begin{proof}
Let $\psi\in H^{1/2}(\Gamma)$. Let $G$ be the fundamental solution of the Laplacian on $\R^d$. According to \cite[Identity 10.5]{Fabes}, $\Pi\psi$ is given by
\begin{equation}
\Pi\psi(x) = \int_\Gamma G(x-y)S^{-1}(\psi(y))\,dS,\ x\in\Om,
\end{equation}
where the operator $S$ is as defined in \cite[p.10]{Fabes}.\\
A routine computation leads to
\[
c_1:=\sup_{x\in\Om}\int_\Gamma G(x-y)\,dS<\infty ,\  c_2:=\sup_{y\in\Gamma}\int_\Om G(x-y)\,dx<\infty.
\]
Thus by H\"older inequality we obtain
\[
\int_\Om (\Pi\psi)^2\,dx\leq c_1 c_2\int_\Gamma (S^{-1}(\psi(y)))^2\,dS.
\]
According to \cite[Theorem 8.1]{Fabes}, the operator $S^{-1}$ operates on the whole space $H^{1/2}(\Gamma)$. In particular, $\int_\Gamma (S^{-1}(\psi(y)))^2\,dS<\infty$, which completes the proof.
\end{proof}
\begin{lem}
Let $\lam\in\R\setminus\sigma_e(L_\Dir)$. Then there is a finite constant $c=c(\lam)>0$ such that
\[
\cel[\psi]\geq -c\int_\Gamma \psi^2\,dS,\ \forall\,\psi\in H^{1/2}(\Gamma).
\]
\label{lsb-Lip}
\end{lem}
\begin{proof}
According to the representation formula from  Theorem \ref{decomposition1} we have
\[
\cel[\psi] =\ccalE[\psi] -\lam(L_\Dir(L_\Dir - \lam)^{-1}\Pi\psi,\Pi\psi)_{L^2(\Om)},\ \forall\,\psi\in H^{1/2}(\Gamma).
\]
Hence the inequality of the lemma is automatically  satisfied for $\lam\leq 0$.\\
Let $\lam>0$. Let us quote that $\ccalE\geq 0$. Therefore,  owing to the representation formula together with Lemma \ref{Pi-continuity} and Cauchy-Schwartz inequality we obtain
\begin{align}
\cel[\psi]\geq -\lam\|L_\Dir(L_\Dir-\lam)^{-1}\| \|\Pi\psi\|^2_{L^2(\Om)}\geq -c\|\psi\|^2_{L^2(\Gamma)},\ \forall\,\psi\in H^{1/2}(\Gamma),
\end{align}
which was to be proved.
\end{proof}
\begin{prop}
The form $\cel$ is closed.
\end{prop}
\begin{proof}
{\em Step 1: $\lam = 0$}. We shall use Theorem \ref{closed-Positive}. Let $(u_n)\subset\Har$ be $\calE^{J}$-Cauchy. As $\calE$ is positive we get $\|Ju_n-Ju_m\|_{L^2(\Gamma)}\to 0$ and hence $\calE[u_n-u_m]\to 0$. Recalling that $\Pi Ju_n=u_n$, making use of Lemma (\ref{Pi-continuity}) we achieve $\|u_n-u_m\|_{L^2(\Om)}\to 0$. Thus there is $u\in H^1(\Om)$ such that $u_n\to u$ in $H^1(\Om)$. Moreover, for any $v\in C_c^\infty(\Om)$ we get
\[
\calE(u_n,v)\to \calE(u,v).
\]
As $(u_n)\subset\Har$ we obtain that $u\in\Har$ as well and then $(\Har,\calE^J)$ is a Hilbert space. According to Theorem \ref{closed-Positive},  $\ccalE$ is closed.\\
{\em Stem 2: General $\lam$}. We use the representation formula together with Theorem \ref{D-to-N-general}. Let $(u_n)\subset \Har(\lam)$ be $\calE_\lam^{(1+cJ)}$-Cauchy. Then  $(Ju_n)$ is a Cauchy sequence in $L^2(\Gamma)$ and $\calE_\lam[u_n-u_m]\to 0$. By Lemma \ref{Pi-continuity} we get that $\Pi Ju_n=Pu_n$ is a Cauchy sequence in $L^2(\Om)$. Now the representation formula yields
\[
\calE_\lam[u_n-u_m]=\cel[Ju_n-Ju_m]=\ccalE[Ju_n - Ju_m] -\lam(L_\Dir(L_\Dir-\lam)^{-1}\Pi J(u_n - u_m),\Pi J(u_n - u_m)).
\]
Hence  $(Ju_n)$ is $\ccalE_1$-Cauchy. By the first step, there is $u\in\Har$ such that
\[
\calE[Pu_n - u] + \|Ju_n - Ju\|^2_{L^2(\Gamma)}\to 0.
\]
Additional use of the representation formula together with Lemma \ref{Pi-continuity}  lead to
\begin{align}
\calE_\lam[u_n - u_\lam]&=\cel[Ju_n - Ju]=\calE[Pu_n - u]\nonumber \\
&-\lam(L_\Dir(L_\Dir-\lam)^{-1}\Pi J(u_n - u),\Pi J(u_n - u))\to 0.
\end{align}
As $u_\lam\in\Har(\lam)$ and $Ju_\lam = Ju$, we get $\calE_\lam[u_n - u_\lam] + (1+c)\|Ju_n - Ju_\lam\|^2_{L^2(\Gamma)}\to 0$. Thereby $(\Har(\lam),\calE^{(1+ cJ)})$ is a Hilbert space and $\cel$ is closed.
\end{proof}
Here $\check L_\lam$ is the D-to-N operator with respect to the boundary $\Gamma$.\\
We close this subsection with a compactness result which was already proved in \cite[Theorem 3.1]{Arendt-Mazzeo}. Here we give a new proof.
\begin{prop}
It holds
\begin{enumerate}
\item The operator $J$ is compact.
\item For every $\lam\in\R\setminus\sigma_e(L_\Dir)$ the operator $\check{L}_\lam$ has compact resolvent.
\end{enumerate}
\label{Compact}
\end{prop}
\begin{proof}
According to \cite[Example 3, p.30]{Jonsson-Wallin} the measure $dS$ is a $(d-1)$-measure, i.e.: for some $c_1, c_2$,
\[
c_1 r^{d-1}\leq \int_{B(x,r)} \,dS\leq c_2r^{d-1},\ \forall\,x\in\Gamma,\ 0<r\leq 1.
\]
By \cite[Lemma 6.1]{Benamor07} there is finite constant $c$ such that for $d\geq 3,\ 2< p\leq \frac{2(d-2)}{d-1}$ and for $d=2,\ p\geq 2$ the following inequality is true
\begin{equation}
(\int_\Gamma |u|^p\,dS)^{2/p}\leq c(\int_{\R^d} |\nabla u|^2\,dx + \int_{\R^d} u^2\,dx),\ \forall\,u\in H^1(\R^d).
\label{Sob-Mix}
\end{equation}
On the other hand, according to \cite[Theorem 5, p. 181]{Stein} there is bounded linear extension operator for $H^1(\Om)$ into $H^1(\R^d)$. Thus the latter inequality holds on $H^1(\Om)$ which in turn, according to \cite[Theorem 7.1]{Benamor07} yields compactness of $J$.\\
By \cite[Theorem 2.10]{BBST} compactness of $J$ yields compactness of the resolvent of $\check L$.\\
Let $(\psi_k)\subset\ran J$ be such that $\sup_k (\cel[\psi_k] + (1+c)\|\psi_k\|_{L^2(\Gamma)}^2)<\infty$, $c=c(\lam)$ is a lower bound for $\cel$. By the latter formula together with inequality (\ref{Pi-continuity}) we conclude that $\sup_k(\ccalE[\psi_k] +\|\psi_k\|_{L^2(\Gamma)}^2)<\infty$. As $(\check L +1)^{-1}$ is compact there is a subsequence $(\psi_{k_j})$ and $\psi\in L^2(\Gamma)$ such that $\psi_{k_j}\to\psi$ in $L^2(\Gamma)$. This means that the embedding $(\ran J, \cel +1+c)\to L^2(\Gamma)$ is compact which is in turn equivalent to compactness of $(\check{L}_\lam +1+c)^{-1}$.

\end{proof}
\subsection{Asymptotic and positivity}
Regarding positivity property for $\ctl$ the problem is completely solved for the unit disc. Whereas for bounded Lipschitz domains it is proved in \cite{Arendt-Mazzeo} that $\ctl$ is p.p. for $\lam<E_0$.\\ Having the theoretical results from the former sections in hands we shall show that the latter property holds in general.
\begin{prop}
The following assertions are true.
\begin{enumerate}
\item For every $\lam\leq 0$, the semigroup $\ctl$ is sub-Markovian, i.e. $\ctl 1\leq 1$ for any $t>0$.
\item For every $\lam< 0$, the semigroup $\ctl$ is ultracontractive:
\[
\ctl : L^2(\Gamma)\to L^\infty(\Gamma)\ \text{is bounded},\ \forall\,t>0.
\]
\item For every $\lam< E_0$, the semigroup $\ctl$ is p.p.
\end{enumerate}
\end{prop}
\begin{proof}
As for $\lam\leq 0$ the form $\calE_\lam$ is a Dirichlet form, the first assertion follows from \cite{BBST} and the fact that Dirichlet forms have sub-Markovian semigroups.\\
Observe that for $\lam<0$ the scalar products $\calE_{\lam}$ are equivalent on $H^1(\Om)$. Thus, using inequality (\ref{Sob-Mix}) together with Dirichlet principle we obtain a Sobolev type inequality: for some finite constant $c=c(\lam)$ we have
\begin{equation}
(\int_\Gamma |\psi|^p\,dS)^{2/p}\leq c \cel[\psi], \forall\,\psi\in H^{1/2}(\Gamma).
\label{Sob}
\end{equation}
It is well known that (see \cite[p.75]{Davies-book}) Sobolev type inequality with $p>2$ together with Dirichlet property for $\cel$ lead to ultracontractivity.\\
The third assertion follows from Lemma \ref{O-p.p.}.
\end{proof}
Now we proceed to establish necessary and sufficient conditions for p.p. near Dirichlet eigenvalues. These conditions will be enlightened by the asymptotic of $\ccalE_\lam$ in this special situation. In particular we shall show that near Dirichlet eigenvalues positivity depends solely on the behavior of either the Dirichlet eigenfunctions or their normal derivatives.\\
Let us first write the representation formula of $\ccalE_z$ for this particular case. To that end,  we denote by $G_\Om$ the Green kernel of the Dirichlet Laplacian on $\Om$ and
\begin{eqnarray*}
 K:=L^2\to L^2,\ Ku = \int_\Om G_\Om(\cdot,y)u(y)\,dy.
\end{eqnarray*}
Let $z\in\R\setminus\sigma_e(L_\Dir)$ and $\psi\in H^{1/2}$ be given. Consider $v,\ w_z$ solutions of
\begin{eqnarray}
\left\{\begin{gathered}
-\Delta v =0, \quad \hbox{in } \Om,\\
v= \psi,~~~{\rm on}\  \Gamma
\end{gathered}
\right.
\end{eqnarray}
and
\begin{eqnarray}
\left\{\begin{gathered}
-\Delta w_z -z w_z=zv, \quad \hbox{in } \Om,\\
w_z= 0,~~~{\rm on}\  \Gamma
\end{gathered}
\right.
\end{eqnarray}
Then
\begin{eqnarray}
v=\Pi\psi,\ w_z -z Kw_z= z Kv,\ {\rm and}\  P_z u = v + w_z.
\end{eqnarray}
Since $z\notin\sigma_e(-\Delta_\Om)$, the operator $(1-z K)$ is invertible and
\begin{eqnarray}
w_z = z(1-z K)^{-1}Kv  = z K(1-z K)^{-1}v.
\end{eqnarray}
Thus in this situation  the representation formula takes the form
\begin{equation}
\check\calE_z[\psi]= \check{\calE}[\psi] - z\int_\Om \Pi\psi(1-z K)^{-1}\Pi\psi\,dx.
\label{decomposition-Lipschitz}
\end{equation}
For bounded Lipschitz domains, we shall show that the Mittag--Leffler expansion for the trace form has a simpler expression involving the trace form at $z=0$, the eigenvalues and the normal derivatives of the eigenfunctions of the Dirichlet Laplacian.\\
Henceforth, we designate by $\nu$ the outward normal unit vector on $\Gamma$.\\
We recall that according to \cite[Theorems 1.1-1.3]{Jerison-Kenig}, if $u\in H^1(\Om)$ and $\Delta u\in L^2(\Om)$ then $u\in H^{3/2}(\Om)$. By \cite[Theorem 1, p.8]{Jonsson-Wallin} we obtain $u|_\Gamma\in H^1(\Gamma)$. Thus the uniqueness part of  \cite[Theorem 10.1]{Fabes} leads to $\frac{\partial u}{\partial\nu}|_\Gamma\in L^2(\Gamma)$.\\
Moreover, owing to \cite[Lemma 14.4]{Tartar} the following version of Green's formula occurs
\[
\int_\Om (-\Delta u) v\,dx= \int_\Om \nabla u\nabla v\,dx  - \int_\Gamma \frac{\partial u}{\partial\nu} v\,dS,\ \forall\,u\in H^{3/2}(\Om),\ v\in H^{1}(\Om).
\]
\begin{prop}
Let  $z\in\C\setminus\sigma_e(L_\Dir)$. Then for any $\psi\in H^{1/2}(\Gamma)$ it holds
\[
\ccalE_z[\psi] = \ccalE[\psi]  + \sum_{k=0}^\infty \frac{z}{E_k(z-E_k)}\big(\int_\Gamma \frac{\partial u_k}{\partial\nu}\psi\,dS)^2.
\]
\label{ML-Lipschitz}
\end{prop}
\begin{proof}
We claim that for all $\psi\in H^{1/2}(\Gamma)$ and $k$ we have
\begin{equation}
(\Pi\psi,u_k)= -\frac{1}{E_k}\int_\Gamma \frac{\partial u_k}{\partial\nu}\psi\,dS.
\label{PiProduct}
\end{equation}
Once identity (\ref{PiProduct}) has been proved, the result would follow from  Theorem \ref{ML}.\\
Let us prove (\ref{PiProduct}). From the above discussion we have $\Pi\psi,u_k\in H^{3/2}(\Om)$. Thus utilizing Green formula we obtain
\begin{align*}
\int_\Om (-\Delta u_k)\cdot \Pi\psi\,dx&=E_k\int_\Om u_k \Pi\psi\,dx
=\int_\Om\nabla u_k\cdot\nabla \Pi\psi\,dx - \int_\Gamma \frac{\partial u_k}{\partial\nu}\psi\,dS\\
&=\int_\Om (-\Delta \Pi\psi)\cdot u_k\,dx + \int_\Gamma \frac{\partial \Pi\psi}{\partial\nu} u_k\,dS
-\int_\Gamma \frac{\partial u_k}{\partial\nu}\psi\,dS\\
&=-\int_\Gamma \frac{\partial u_k}{\partial\nu}\psi\,dS.
\end{align*}
This leads to
\[
(\Pi\psi,u_k) = -E_k^{-1}\int_\Gamma \frac{\partial u_k}{\partial\nu}\psi\,dS,
\]
and the claim is proved.

\end{proof}
\begin{theo}
Let $\lam>E_0$ and $E$ be an eigenvalue of the Dirichlet Laplacian with multiplicity $m$. Let $(v_1,\cdots,v_m)$ be an orthonormal basis for $\ker(L_\Dir-E)$. Then $\ctl,\ t>0$ is p.p. on a left (resp. right) neighborhood of $E$ if and only if one of the following conditions is fulfilled:
\begin{enumerate}
\item
\[
\sum_{k=1}^m (\Pi\psi^+,v_k)\cdot(\Pi\psi^-,v_k)\geq 0,\ (\text{resp.} \leq 0), \forall\,\psi\in L^2(\Gamma).
\]
\item
\[
\sum_{k=1}^m (\int_\Gamma\frac{\partial v_k}{\partial\nu}\psi^+\,dS)\cdot(\int_\Gamma \frac{\partial v_k}{\partial\nu} , \psi^-\,dS)
\geq 0,\ (\text{resp.} \leq 0), \forall\,\psi\in L^2(\Gamma).
\]
\end{enumerate}

It follows, in particular that if there is $\psi_1,\psi_2\in L^2(\Gamma)$ are such that the above corresponding sums have opposite signs then $\ctl$ is not p.p. in any neighborhood of $E$.
\label{Pos-Multiple2}
\end{theo}
\begin{proof}
Let $\lam$, $E$ and $(v_1,\cdots,v_m)$ be as in the theorem and $\psi\in H^{1/2}(\Gamma)$. From Green formula we get
\[
(\Pi\psi^+,v_k)\cdot(\Pi\psi^-,v_k) = \frac{1}{E^2}(\int_\Gamma \frac{\partial v_k}{\partial\nu}\psi^+\,dS)
\cdot(\int_\Gamma \frac{\partial v_k}{\partial\nu}\psi^-\,dS).
\]
Hence both sums appearing in the statement of the theorem are equal up to the positive factor $1/{E^2}$.
By Theorem \ref{Positivity-Multiple} positivity holds if and only if one of the equivalent  conditions $1-2$ hold for every $\psi\in H^{1/2}(\Gamma)$. Finally, the continuity of the scalar product together with the fact that $H^{1/2}(\Gamma)$ is dense in $L^2(\Gamma)$ gives the result, which completes the proof.
\end{proof}
\begin{coro}
Let $E$ be a simple eigenvalue of $L_\Dir$ with associated normalized eigenfunction $u_E$. Assume that $u_E$ or $\frac{\partial u_E}{\partial\nu}$ can be chosen to have constant sign. Then $\ctl$ is not p.p. in any  right neighborhood of $E$ whereas it is p.p in $\{\lam\colon\,-r_E<\lam-E<0\}$.
\label{SignNormalDeriv}
\end{coro}
\begin{proof}
If $u_E$ has a constant sign, owing to positivity of $\Pi$ we get $(\Pi\psi^+,u_E)\cdot(\Pi\psi^-,u_E)\geq 0$. Whereas if $\frac{\partial u_E}{\partial\nu}$ has constant sign we obtain $(\int_\Gamma \frac{\partial u_E}{\partial\nu}\psi^+\,dS)\cdot(\int_\Gamma \frac{\partial u_E}{\partial\nu}\psi^-\,dS)\geq 0$, and the result follows from Theorem \ref{Pos-Multiple2}.
\end{proof}
\subsection{The D-to-N on the unit disc revisited}

For the unit disc the eigenvalues of the Dirichlet Laplacian are $j_{k,l}^2,\ k,l\in\N_0:=\N\cup\{0\}$ where $j_{k,l}$ are the positive zeros of the Bessel functions $J_k$. They are either simple or double eigenvalues. The set of simple eigenvalues consists of $\{j_{0,l}^2,\ l\in \N_0\}$ with associated normalized eigenfunctions
\[
u_{0,l}(r):= c_lJ_0(j_{0,l}r),\ c_l=  \frac{\sqrt 2}{ |J_0'(j_{0,l})| }.
\]
The set of double eigenvalues consists of $\{j_{k,l}^2,\ k\in\N,l\in \N_0 \}$ with associated normalized eigenfunctions
\begin{eqnarray}
u_{k,l}(r):= c_{k,l}J_k(j_{k,l}r)\cos(k\theta),\ v_{k,l}(r):= c_{k,l}J_k(j_{k,l}r)\sin(k\theta),\   c_{k,l}=\frac{\sqrt 2}{\sqrt{\pi} |J_k'(j_{k,l})| }.
\label{Normalize}
\end{eqnarray}
Let us mention that for $v_{k,l}$ the integer $k$ runs $\N$.\\
In order to compute Mittag--Leffler expansion we have to compute the normal derivatives of the eigenfunctions. Obviously if $E$ is an eigenvalue of the Dirichlet Laplacian on the unit disc with associated normalized eigenfunction $u_E$ then
\[
\frac{\partial u_E}{\partial\nu} = \frac{\partial u_E}{\partial r}|_{r=1}= c\sqrt E J_0'(\sqrt E).
\]
Thus
\begin{equation}
 \frac{\partial u_{0,l}}{\partial\nu}= c_l j_{0,l} J_0'(j_{0,l}),\ \frac{\partial u_{k,l}}{\partial\nu}
 =c_{k,l} j_{k,l} J_k'(j_{k,l})\cos(k\theta),\
 \frac{\partial v_{k,l}}{\partial\nu} = c_{k,l} j_{k,l} J_k'(j_{k,l})\sin(k\theta).
\label{Noraml-Deriv}
\end{equation}
\begin{prop}
Let $\psi\in H^{1/2}(\Gamma)$. Then
\begin{align*}
\ccalE_z[\psi] &=\ccalE[\psi] + \sum_{l=0}^\infty \frac{2z}{ (z- j_{0,l}^2)} (\int_0^{2\pi} \psi(\theta)\,d\theta)^2
+ \sum_{k=1,l=0}^\infty \frac{2z}{  \pi (z- j_{k,l}^2)} (\int_0^{2\pi} \cos(k\theta)\psi(\theta)\,d\theta)^2\\
&+ \sum_{k=1,l=0}^\infty \frac{2z}{ \pi (z- j_{k,l}^2)} (\int_0^{2\pi} \sin(k\theta)\psi(\theta)\,d\theta)^2.
\label{ForMLDisc}
\end{align*}
\label{MLDisc}
\end{prop}
\begin{proof}
Let $\psi\in H^{1/2}(\Gamma)$.  Making use of Mittag-Leffler expansion from Proposition \ref{ML-Lipschitz} together with the normal derivatives (\ref{Noraml-Deriv}) we obtain
\begin{align}
\ccalE_z[\psi] &=\ccalE[\psi] + z\sum_l \frac{1}{j_{0,l}^2 (z- j_{0,l}^2)} (c_l j_{0,l} J_0'(j_{0,l}))^2 (\int_0^{2\pi} \psi(\theta)\,d\theta)^2\nonumber\\
&+ z\sum_{k=1,l=0}^\infty \frac{1}{j_{k,l}^2 (z- j_{k,l}^2)}(c_{k,l} j_{k,l} J_k'(j_{k,l}))^2 (\int_0^{2\pi} \cos(k\theta)\psi(\theta)\,d\theta)^2\nonumber\\
&+ z\sum_{k=1,l=0}^{\infty} \frac{1}{j_{k,l}^2 (z- j_{k,l}^2)}(c_{k,l} j_{k,l} J_k'(j_{k,l}))^2 (\int_0^{2\pi} \sin(k\theta)\psi(\theta)\,d\theta)^2
\end{align}
Taking the expression of $c_l$ and $c_{k,l}$ from (\ref{Noraml-Deriv}) into account leads to the sought formula.
\end{proof}
The eigenvalues as well as the eigenfunctions  of $\check{L}_z$ were computed in \cite{Daners}.
The eigenvalues are
\[
\check{E}_k(z) = k - \frac{\sqrt\lam J_{k+1}(\sqrt z)}{J_k(\sqrt z)},\ k\in\N_0,
\]
with normalized eigenfunctions $\psi_k$:
\[
\frac{1}{\sqrt{2\pi}},\ \frac{1}{\sqrt \pi}\cos(k\theta),\ \frac{1}{\sqrt\pi}\sin(k\theta),\ k\in\N.
\]
We quote that $\check{E_0}(\lam)$ is the sole simple eigenvalue whereas all others are double eigenvalues. Moreover the eigenfunctions are $z$-independent.\\
Let us quote that one can rediscover the values of the eigenfunctions and the eigenvalues of $\check{L}_z$ on the light of formula (\ref{ForMLDisc}).\\
We shall use Mittag--Leffler expansion for $\ccalE_z$ from Proposition (\ref{MLDisc}) to write Mittag--Leffler expansion for $\check{E}_k(z)$.
\begin{prop}
Let $\check{E}_k(z)$ be an eigenvalue of $\check{L}_\lam$ with eigenfunction $\psi_k$. Then
\begin{eqnarray}
\check{E}_k(z) = k + \sum_{l=0}^\infty \frac{2z}{z- j_{k,l}^2},
\label{EigenDisc}
\end{eqnarray}
where the series converges absolutely.
\end{prop}
We mention that formula (\ref{EigenDisc}) was established in \cite{Daners}, using properties of Bessel functions. We shall give an other proof.
\begin{proof}
In case $k=0$, then $\psi_0=\frac{1}{\sqrt{2\pi}}$. Hence by formula (\ref{ForMLDisc}) we get
\[
\check{E}_0(z) =  \sum_{l=0}^\infty \frac{2z}{ (z- j_{0,l}^2)}.
\]
In case $k\neq 0$, then either $\psi_k(\theta)= \frac{1}{\sqrt \pi}\cos(k\theta)$ or $\psi_k(\theta)=\frac{1}{\sqrt\pi}\sin(k\theta)$. Thus in both cases we have $\int_0^{2\pi} \psi_k\,d\theta=0$. Whereas in case $\psi_k(\theta)= \frac{1}{\sqrt \pi}\cos(k\theta)$ we get
\[
\int_0^{2\pi} \cos(n\theta)\psi_k(\theta)\,d\theta= \sqrt{\pi}\delta_{n,k},\    \int_0^{2\pi} \sin(n\theta)\psi_k(\theta)\,d\theta=0.
\]
Similar formulae hold in case  $\psi_k(\theta)= \frac{1}{\sqrt \pi}\sin(k\theta)$. Thus in both case we get
\[
\check{E}_k(z) = \ccalE[\psi_k] + \sum_{l=1}^\infty \frac{2z}{z - j_{k,l}^2} = k + \sum_{l=1}^\infty \frac{2z}{z - j_{k,l}^2},
\]
which completes the proof.

\end{proof}
We turn our attention to analyze positivity of $\ctl$ for $\lam>E_0$.
\begin{prop}
Let $E$ be a simple Dirichlet eigenvalue of $L_\Dir$. Then the semigroup is p.p. on every small left neighborhood of $E$ whereas it is non-p.p. on any right neighborhood of $E$.
\end{prop}
\begin{proof}
For the unit disc if $E$ is a simple eigenvalue then the associated normalized eigenfunction is radially symmetric and is of the type
\[
u_E(r)=cJ_0(\sqrt E r).
\]
Plainly
\[
\frac{\partial u_E}{\partial\nu} = \frac{\partial u_E}{\partial r}|_{r=1}= c\sqrt E J_0'(\sqrt E)\neq 0,
\]
because $J_0$ has only simple positive zeros. Hence $\frac{\partial u_E(r)}{\partial\nu}$ has constant sign and  on the light of Corollary \ref{SignNormalDeriv} we get the result.
\end{proof}
The latter proposition was established  in \cite[Theorem 1.1]{Daners}, however with a different proof.\\
Let us now turn our attention to double eigenvalues. Let $E$ be a double eigenvalue for $L_\Dir$ and $(u_1,u_2)$ be an orthonormal basis for $\ker (L_\Dir -E)$. Then
\[
u_1(r,\theta)=c_1J_m({\sqrt E}r)\cos(m\theta),\  u_2(r,\theta)=c_2J_m({\sqrt E}r)\sin(m\theta).
\]
Hence
\[
\frac{\partial u_1}{\partial\nu} = \frac{\partial u_1}{\partial r}|_{r=1}= -mc_1{\sqrt E} J_m'(\sqrt E)\sin(m\theta),\  \frac{\partial u_2}{\partial\nu} = \frac{\partial u_2}{\partial r}|_{r=1}= mc_2{\sqrt E} J_m'(\sqrt E)\cos(m\theta).
\]
They both change signs many times. Thus we are led to use the last  assertion of Theorem \ref{Pos-Multiple2} to handle the question.\\
The following result was already proved in \cite{Daners}. We shall prove it by using our method.
\begin{prop}
The semigroup $\ctl$ is non p.p. on any neighborhood of any double eigenvalue.
\end{prop}
\begin{proof}
Take $\psi_k= \frac{\partial u_k}{\partial\nu},\ k=1,2$. Then
\[
\int_\Gamma\psi_k \psi_k^+\,dS = \int_\Gamma (\psi_k^+)^2\,dS>0,\  \int_\Gamma\psi_k \psi_k^-\,dS
= - \int_\Gamma (\psi_k^-)^2\,dS<0.
\]
On the other hand, owing to the $L^2(\Gamma)$ orthogonality of $\psi_1,\psi_2$ we get
\begin{align*}
\int_\Gamma\psi_2 \psi_1^+\,dS = \int_\Gamma\psi_2 \psi_1^-\,dS.
\end{align*}
A straightforward computation leads to (up to a constant)
\begin{align*}
\int_\Gamma\psi_2 \psi_1^+\,dS &= \int_0^{2\pi}  \cos(m\theta)\sin^+(m\theta)\,d\theta=
 \frac{1}{m}\int_0^{2m\pi}  \cos(\theta)\sin^+(\theta)\,d\theta\\
 &=\frac{1}{m} \sum_{k=0}^{2m-1} \int_{k}^{(k+1)\pi}  \cos(\theta)\sin^+(\theta)\,d\theta
 = \frac{1}{m}\sum_{k=0}^{m-1} \int_{2k\pi}^{(2k+1)\pi}  \cos(\theta)\sin(\theta)\,d\theta=0.
 \end{align*}
Putting all together we obtain
\[
\sum_{k=1}^2 \big(\int_\Gamma \frac{\partial u_k}{\partial\nu}\psi_1^+\,dS\big)\cdot\big(\int_\Gamma \frac{\partial u_k}{\partial\nu}\psi_1^-\,dS\big)
<0.
\]
Thus, according to Theorem \ref{Pos-Multiple2}, $\ctl$ is non p.p. on a left neighborhood of $E$.\\
To show that $\ctl$ is non p.p. on a right neighborhood of $E$ we take $\psi_m(\theta)=\cos(m\theta)$. Set
\[
I_1^\pm=\int_\Gamma \cos((m+1)\theta)\psi^\pm(\theta)\,d\theta,\
I_2^\pm=\int_\Gamma \sin((m+1)\theta)\psi^\pm(\theta)\,d\theta.
\]
A lengthy computation leads to
\[
I_1^+= \frac{2m+3}{2(2m+1)}\sin(\frac{\pi}{2m}) - \frac{1}{2}.
\]
On can easily check that $I_1^+\neq 0$. By orthogonality consideration we obtain $I_1^-=I_1^+$ and hence  $I_1^-I_1^+>0$.\\
Similarly we get $I_2^+=I_2^-$ and then  $I_2^+I_2^-\geq 0$. Summarizing we get
\[
\sum_{k=1}^2 \big(\int_\Gamma \frac{\partial u_k}{\partial\nu}\psi_{m-1}^+\,dS\big)\cdot\big(\int_\Gamma \frac{\partial u_k}{\partial\nu}\psi_{m-1}^-\,dS\big)
>0.
\]
Once again, according to Theorem \ref{Pos-Multiple2}, $\ctl$ is non p.p. on a right  neighborhood of $E$.

\end{proof}

\subsection{The D-to-N on the square}

Compared to the unit disc we shall see that for the square the picture changes drastically regarding positivity. In fact, we shall show that p.p. fails to hold true in any neighborhood of any Dirichlet  eigenvalue except the smallest one. Therefore regularity of the boundary affect positivity of the semigroup.\\
We consider the square $\Om=(0,1)\times(0,1)$. It is known that the eigenvalues of the Dirichlet Laplacian on the square  are
\[
E_{m,n}=\pi^2(m^2 + n^2),\ m,n\in\N,
\]
with associated normalized eigenfunctions
\[
u_{m,n}(x,y)=2\sin(m\pi x)\sin(n\pi y).
\]
The eigenvalues are either simple or double eigenvalues. Moreover, the eigenfunction associated to the smallest eigenvalue $E_0:=E_{1,1}$ can be chosen to be positive.\\
The normal derivatives of $u_{m,n}$ are, respectively
\begin{eqnarray*}
\frac{\partial u_{m,n}}{\partial\nu} = 2\pi\cdot\left\{\begin{gathered}
 -n \sin(m\pi x)\ \text{on}\ (0,1)\times\{0\}\\
(-1)^{m} m \sin(n\pi y)\ \text{on}\ \{1\}\times(0,1)\\
n (-1)^{n} \sin(m\pi x)   \ \text{on}\  (0,1)\times\{1\}\\
-m \sin(n\pi y)\ \text{on}\ \{0\}\times(0,1)
\end{gathered}.
\right .
\end{eqnarray*}
They are in $L^2(\Gamma)$ and all change sign except for $m=n=1$.\\
For  simple eigenvalues $E_{m,m},\ m\geq 2$ the normal derivatives are
\begin{eqnarray*}
\frac{\partial u_{m,m}}{\partial\nu} = 2m\pi\cdot\left\{\begin{gathered}
 - \sin(m\pi x)\ \text{on}\ (0,1)\times\{0\}\\
(-1)^{m}  \sin(m\pi y)\ \text{on}\ \{1\}\times(0,1)\\
 (-1)^{m} \sin(m\pi x)   \ \text{on}\  (0,1)\times\{1\}\\
-m \sin(m\pi y)\ \text{on}\ \{0\}\times(0,1)
\end{gathered}.
\right .
\end{eqnarray*}
We know from Proposition \ref{Positivity-Under} that $\ctl$ is p.p. to the left $E_{1,1}$. Thus by Theorem \ref{Pos-Multiple2} it is non p.p. on  a left neighborhood of $E_{1,1}$.
\begin{prop}
Let $E_{m,n}>E_0$ be an eigenvalue of the Dirichlet Laplacian on the square. Then $\ctl$ is non p.p. in any neighborhood of $E_{m,n}$.
\end{prop}
\begin{proof}
Let $E_{m,m}$ be a simple eigenvalue of $L_\Dir$ with $m\geq 2$.\\
For any $\psi\in L^2(\Gamma)$, we set
\[
I^\pm(\psi)=\int_\Gamma \frac{\partial u_{m,m}}{\partial\nu}\psi^\pm\,dS.
\]
{\em First case:} $m$ is even. For the  choice $\psi = \frac{\partial u_{m,m}}{\partial\nu}$ owing to the fact that the normal derivative changes sign  we get
\begin{equation}
I^+(\psi)I^-(\psi)<0.
\label{Step1}
\end{equation}
Let $\psi\in L^2(\Gamma)$. An elementary computation shows that for even $m$ we have
\begin{align*}
I^\pm(\psi) &=\int_\Gamma \frac{\partial u_{m,m}}{\partial\nu}\psi^\pm\,dS=
2m\pi \int_0^1 \sin(m\pi x) (\psi^\pm(1,x) - \psi^\pm(x,1))\,dx\\
&+ 2m\pi  \int_0^1 \sin(m\pi x) (\psi^\pm(0,x) - \psi^\pm(x,0))\,dx.
\end{align*}
We consider the case $m=2$, the proof general even $m$ is similar. Let us choose now
\[
\psi(x,0)=\cos(\pi x),\ \psi(1,x)=x,\ \psi(x,1)=\cos(\pi x),\ \psi(0,x)=x.
\]
Then with this choice of $\psi$ we obtain (up to a factor)
\begin{align*}
I^+(\psi)&=\int_0^{1/2} \sin(2\pi x)(x-\cos(\pi x))\,dx + \int_{1/2}^1 x\sin(2\pi x)\,dx
\\
&+ \int_0^{1/2} \sin(2\pi x)(x-\cos(\pi x))\,dx + \int_{1/2}^1 x\sin(2\pi x)\,dx\\
 &= 2\int_0^1 x\sin(2\pi x)\,dx - 2\int_0^{1/2} \sin(2\pi x)\cos(\pi x)\,dx\\
&= -\frac{1}{2\pi} - \frac{4}{3\pi}.
\end{align*}
Whereas
\[
I^-(\psi)= 2\int_{1/2}^1 \sin(2\pi x)\cos(\pi x)\,dx = -\frac{4}{3\pi}.
\]
Thus
\begin{equation}
I^+(\psi) I^-(\psi) > 0
\label{Step2}
\end{equation}
Inequalities \ref{Step1}-\ref{Step2} in conjunction with Theorem \ref{Pos-Multiple2} yield that $\ctl$ is non p.p. in any neighborhood of $E_{m,m}$.\\
{\em Second case:} $m$ is odd. Choosing $\psi$ as in the first step we get $I^+(\psi) I^-(\psi)<0$.\\
When choosing
\[
\psi(x,0)=1=\psi(1,x),\ \psi(x,1)=\psi(0,x)=-1,
\]
we get $I^+(\psi)=I^-(\psi)=4$ and then $I^+(\psi)I^-(\psi)>0$. Finally an application of Theorem \ref{Pos-Multiple2} yields the first assertion.\\
Now let $E_{m,n}$ be a double eigenvalue with eigenfunctions $v_1,v_2$. As $\frac{\partial v_1}{\partial\nu}$ changes sign, taking $\psi_1= \frac{\partial v_1}{\partial\nu}$ we obtain $\int_\Gamma \psi_1 \psi_1^+\,dS>0$ and  $\int_\Gamma \psi_1 \psi_1^-\,dS<0$.\\
Let us emphasize that  $\frac{\partial v_1}{\partial\nu}, \frac{\partial v_2}{\partial\nu}$ are $L^2(\Gamma)$-orthogonal. Thereby,   $\int_\Gamma  \frac{\partial v_2}{\partial\nu}\psi_1^+\,dS=\int_\Gamma  \frac{\partial v_2}{\partial\nu}\psi_1^-\,dS$. Furthermore an elementary computations leads to $\int_\Gamma  \frac{\partial v_2}{\partial\nu}\psi_1^\pm\,dS=0$.\\
Owing to these considerations we achieve
\[
\sum_{k=1}^2 (\int_\Gamma  \frac{\partial v_k}{\partial\nu}\psi_1^+\,dS)\cdot (\int_\Gamma\frac{\partial v_k}{\partial\nu}\psi_1^-\,dS)< 0.
\]
Once again, by Theorem \ref{Pos-Multiple2} we conclude that $\ctl$ is non p.p. to the left of $E_{m,n}$.\\
To prove non positivity to the right we follow the strategy we used for the unit disc. Let us choose
\[
\psi_{m,n}= \frac{\partial u_{m+1,n+1}}{\partial\nu}.
\]
Set
\[
I_1^\pm = \int_\Gamma \frac{\partial u_{m,n}}{\partial\nu} \psi_{m,n}^\pm\,dS,\
I_2^\pm = \int_\Gamma \frac{\partial u_{n,m}}{\partial\nu} \psi_{m,n}^\pm\,dS.
\]
Orthogonality leads to $I_{1,2}^+=I_{1,2}^-$. Moreover, an elementary computation leads to
\begin{align*}
 I_1^+&= 2n(n+1)\int_0^1 \sin(m\pi x)\sin^+((m+1)\pi x)\,dx\\
& -2m(m+1) \int_0^1 \sin(n\pi x)\sin^+((n+1)\pi x)\,dx\\
&= \frac{ mn(n+1)(2m-1)}{(2m+1)\pi}\sin(\frac{\pi}{m+1})
- \frac{ mn(m+1)(2n-1)}{(2n+1)\pi}\sin(\frac{\pi}{n+1}).
\end{align*}
On the other hand the functions
\[
[1,\infty)\to\R,\ x\mapsto \frac{2x-1}{(2x+1)(x+1)}\ \text{and}\ x\mapsto\sin(\frac{\pi}{x+1}),
\]
are both positive and strictly decreasing. Hence $I_1^+\neq 0$. Putting all together we achieve $I_1^+ I_1^- + I_2^+ I_2^->0$. According to Theorem \ref{Pos-Multiple2}, $\ctl$ is non p.p. to the right of $E_{m,n}$ and the proof is finished.

\end{proof}

\subsection{The D-to-N on the unit ball}

We denote by $B$ the unit ball in $\R^3$. The eigenvalues of the Dirichlet Laplacian on $B$ are squares of the positive zeros of the modified spherical  Bessel functions of the first kind
\[
j_n(z):=\sqrt{\frac{\pi}{2z}}J_{n+1/2}(z).
\]
They coincide with squares of positive zeros of $J_{n+1/2}$ and shall be enumerated
\[
j_{nk}^2,\ k,n\in\N_0:=\N\cup\{0\},
\]
with respective multiplicities $2n+1$ and normalized  eigenfunctions
\[
u_{nkl}:= c_{nkl}j_n(j_{nk}r)Y_{nl}(\theta,\varphi),\ |l|\leq n;\
 c_{nkl}=\frac{\sqrt 2}{|j_n'(j_{nk})|}.
\]
where $Y_{nl}$ are the normalized spherical harmonics.\\
The respective normal derivatives are:
\[
\frac{\partial u_{nkl}}{\partial\nu}= \frac{\partial u_{nkl}}{\partial r}|_{r=1}= c_{nkl} j_{nk} j_n'(j_{nk})Y_{nl}(\theta,\varphi).
\]
\begin{prop}
Let $z\in\C\setminus\sigma_e(L_\Dir)$ and $\psi\in H^{1/2}(\Gamma)$. Then
\begin{equation}
\ccalE_z[\psi] = \ccalE[\psi] +  \sum_{k=0}^\infty \sum_{n=0}^\infty \sum_{l=-n}^n \frac{ 2z}{  (z - j_{nk}^2 )} \big|\int_\Gamma Y_{nl}\psi\,dS\big|^2.
\label{MLFormBall}
\end{equation}
\end{prop}
\begin{proof}
According to Proposition \ref{ML-Lipschitz} we have
\[
\ccalE_z[\psi] = \ccalE[\psi] + z \sum_{k=0}^\infty \sum_{n=0}^\infty \sum_{l=-n}^n \frac{1}{ j_{nk}^2 (z- j_{nk}^2 )} \big| \int_\Gamma \frac{\partial u_{nkl}}{\partial\nu}\psi\,dS\big|^2.
\]
Inserting the expression of the normal derivatives of the Dirichlet eigenfunctions in the latter identity gives the result.
\end{proof}
We shall use the series expansion (\ref{MLFormBall}) to derive Mittag--Leffler expansion for the eigenvalues of $\check{L}_z$ .
\begin{theo}
For any $z\in\C\setminus\sigma_e(L_\Dir)$ it holds:
\begin{enumerate}
\item The spherical harmonics are the eigenfunctions of $\check{L}_z$.
\item For each $n\in\N_0$ the spherical harmonics $Y_{nl},\ |l|\leq n$ are normalized eigenfunctions of an eigenvalue   $\check{E}_n(z)$ of $\check{L}_z$. Moreover,
\begin{eqnarray}
 \check{E}_n(z) = n + \sum_{k= 0}^\infty\frac{2z}{z - j_{nk}^2},
 \label{MLBall}
\end{eqnarray}
where the series converges absolutely.
\end{enumerate}
\label{SpecSphere}
\end{theo}
\begin{proof}
We first compute the eigenfunctions and the eigenvalues of $\ccalE$.\\
Making use of spherical coordinates we get that  for any $Y_{nl}$ the solution of
\begin{eqnarray}
\left\{\begin{gathered}
-\Delta u =0, \quad \hbox{in } B,\\
u= Y_{nl},~~~{\rm on}\  \Gamma
\end{gathered}
\right.
\end{eqnarray}
is given by $r^nY_{nl}$. Hence $\Pi (Y_{nl})=r^nY_{nl}$ and $\ccalE(Y_{nl},\psi)=\calE(\Pi Y_{nl},\Pi\psi)=n\int_\Gamma Y_{nl}\psi\,dS$, by Green's formula, for any $\psi\in H^{1/2}(\Gamma)$. This shows that $n$ is an eigenvalue of $\ccalE$ of order $2n+1$ with normalized eigenfunction $Y_{nl},\ |l|\leq n$.\\
Let $\psi\in H^{1/2}(\Gamma)$, by polarization and since the $Y_{nl}$ are an orthonormal basis for $L^2(\Gamma)$ we get
\begin{align*}
\ccalE_z(Y_{mp},\psi)&= \ccalE(Y_{mp},\psi) +  \sum_{k=0}^\infty \sum_{n=0}^\infty \sum_{l=-n}^n \frac{2z}{  (z - j_{nk}^2 )} \big( \int_\Gamma Y_{nl}\psi\,dS\big)  \big( \int_\Gamma Y_{nl} \bar{Y}_{mp}\,dS\big)\\
&=m + \sum_{k=0}^\infty \frac{ 2z }{  (z - j_{mk}^2 )} \big( \int_\Gamma Y_{mp}\psi\,dS\big)
\end{align*}
For every $p\in\N_0$ and every $|m|\leq p$, $Y_{mp}$ is an eigenfunction of $\ccalE_z$ associated to the eigenvalue
\[
m + \sum_{k=0}^\infty \frac{ 2z}{  (z - j_{mk}^2 )},
\]
which  is of order $2m+1$.
\end{proof}
\begin{rk}
{\rm
On the light of Proposition \ref{MLDisc} and Theorem \ref{SpecSphere}, it seems that formula (\ref{MLBall}) should be true for every dimension $d\geq 2$.
}
\end{rk}
Let us now investigate positivity of $\ctl,\ \lam>E_0$. We first record that owing to our general result form Corollary \ref{SignNormalDeriv} the semigroup $\ctl$ is p.p. to the left of $E_0=j_{00}^2$ and non p.p. on a right neighborhood of it.
\begin{prop}
Let $j_{nk}^2$ be an eigenvalue of $L_\Dir$. Then
\begin{enumerate}
\item The semigroup  $\ctl$ is p.p. on a left neighborhood of $j_{0k}$ whereas it is non p.p. on a right neighborhood of  $j_{0k}$.
\item The semigroup  $\ctl$ is non p.p. in any neighborhood of $j_{nk}^2$ for $n\geq 1$.
\end{enumerate}
\end{prop}

\begin{proof}
For $n=0$ the normal derivatives of the  eigenfunctions are constant. Hence by Corollary \ref{SignNormalDeriv}  we get assertion 1.\\
In what follows $c$ designates a nonzero generic constant which may differ from line to line.\\
Assume now that $n\geq 1$. Then the normal derivatives of the real eigenfunctions of $j_{nk}^2$ are
\[
v_l:=cP_n^l(\cos\theta)\cos(l\varphi),\ l=0,\cdots,n ,\ w_l:=cP_n^l(\cos\theta)\sin(l\varphi),\ l=1,\cdots,n.
\]
We recall that
\[
P_n^n(\cos\theta)= c|\sin(\theta)|^n,\ w_n=P_n^n(\theta)\sin(n\varphi).
\]
Let us choose $\psi=w_n$. Then
\[
\int_\Gamma w_n w_w^+\,dS>0,\ \int_\Gamma w_n w_n^-\,dS<0.
\]
Let $l=1,\cdots,n-1$. Then $\int_\Gamma v_lw_n\,dS= \int_\Gamma w_lw_n\,dS =0$. Thus
\[
\int_\Gamma v_lw_n^+\,dS=\int_\Gamma v_lw_n^-\,dS,\ \int_\Gamma w_lw_n^-\,dS=\int_\Gamma w_lw_n^+\,dS.
\]
Owing to the orthogonality of Legendre polynomials we obtain, for any $l=1,\cdots,n-1$
\begin{align*}
\int_\Gamma v_l w_n^+\,dS &= \int_0^\pi P_n^l(\cos\theta)P_n^n(\cos\theta)\sin(\theta)\,d\theta
\cdot\int_0^{2\pi}\cos(l\varphi) \sin^+(n\varphi)\,d\varphi\\
&=0 = \int_\Gamma w_l w_n^+\,dS.
\end{align*}
Consequently, according to Theorem \ref{Pos-Multiple2} we conclude that $\ctl$ is non p.p. on left neighborhoods  of $j_{nk}^2$.\\
Let us show non positivity to the right of $j_{nk}^2$.\\
We recall that
\[
P_n^n(\cos\theta)= c|\sin(\theta)|^n,\ w_n=P_n^n(\theta)\sin(n\varphi).
\]
Arguing as in the former examples it suffices to find a function $\psi$ which is $L^2(\Gamma)$ orthogonal to all $v_l,\ w_l$ and  such that  $\int_\Gamma w_n \psi^+\,dS\neq 0$.\\
Assume that $n$ is even. Let us choose $\psi(\varphi)=\sin((n+1)\varphi)$. By a straightforward computation we obtain (up to a constant)
\begin{align*}
\int_\Gamma w_n \psi^+\,dS&=\int_0^\pi\sin(n\varphi)\sin^+((n+1)\varphi)\,d\varphi + \int_0^\pi\sin(n\varphi)\sin^-((n+1)\varphi)\,d\varphi\\
& = 2\int_0^\pi\sin(n\varphi)\sin^+((n+1)\varphi)\,d\varphi = \frac{n(2n-1)}{2(2n+1)}\sin(\frac{\pi}{n+1})\neq 0.
\end{align*}
For odd $n$ we choose $\psi(\varphi)=\sin((n+2)\varphi)$ and obtain
$\int_\Gamma w_n \psi^+\,dS\neq 0$.\\
Once again according to Theorem \ref{Pos-Multiple2} the semigroup $\ctl$ is non p.p. on right neighborhoods of $j_{nk}^2$.

\end{proof}
Finally, according to elliptic regularity,  let us indicate that our method still works for operators of the type
\[
L u=\sum_{i,j=1}^d \frac{\partial}{\partial x_i}(a_{ij}(x) \frac{\partial u}{\partial x_j} ) + V,
\]
such that the matrix $A(x):=(a_{ij}(x))_{ij}$ is  symmetric positive definite and $V=V^+-V^-$ is measurable such that
\[
V^-\ \text{is in the Kato-class and}\   \int_\Om u^2 V^+\,dx<\infty,\ \forall\,u\in H^1(\Om).
\]

\section{The Robin-Laplacian on Lipschitz domains}

Here we shall be concerned with D-to-N operator related to Robin Laplacian on a domain fulfilling all conditions of the latter section. We also keep the notations from the latter sections.\\
Let $\beta\in C(\Gamma)$ be such that $\beta>0$. We consider the quadratic form $Q$  defined by
\[
Q:\ \dom Q= H^1(\Om),\ Q[u] = \int_\Om |\nabla u|^2\,dx + \int_\Gamma \beta u^2\,dS.
\]
It is well know and can be easily proved that $Q$ is a Dirichlet form in $L^2(\Om)$. The selfadjoint operator related to $Q$ which we denote by $-\Delta_{\rm Rob}$ is commonly named the Robin Laplacian:
\begin{align*}
\dom(-\Delta_{\rm Rob}) = \{ u\in H^1(\Om),\ \Delta u\in L^2(\Om),\
\frac{\partial u}{\partial\nu} +\beta u=0\ \text{on}\ \Gamma\},\ -\Delta_{\rm Rob}u=-\Delta u.
\end{align*}
Let the operator $J$ be as in the latter section, the trace to the boundary operator.\\
Plainly all conditions demanded for the validity of  Theorem \ref{SameBehav} are fulfilled by $Q$ and $\calE$, where $\calE$ is the quadratic form related to the Neumann Laplacian. Thus we can likely construct $\check{Q}_\lam$ and obtain thereby a closed densely defined form in $L^2(\Gamma)$ for any $\lam\in\R\setminus\sigma_e(L_\Dir)$.\\
For $\lam\in\R\setminus\sigma_e(L_\Dir)$,  set $\check{S_t}(\lam)$ the semigroup related to $\check{Q}_\lam$.
Accordingly, the p.p. property for $\check{S_t}(\lam)$ can be easily analyzed on the light of Theorem \ref{SameBehav}. In fact,  a straightforward application of Theorem \ref{SameBehav} leads to
\begin{theo}
All results from the latter section regarding $\ctl$ are still valid for $\check{S_t}(\lam)$.
\end{theo}

\bibliography{BiblioDtoN}

\begin{thebibliography}{BBST19}

\bibitem[AM12]{Arendt-Mazzeo}
W.~Arendt and R.~Mazzeo.
\newblock Friedlander's eigenvalue inequalities and the
  {D}irichlet-to-{N}eumann semigroup.
\newblock {\em Commun. Pure Appl. Anal.}, 11(6):2201--2212, 2012.

\bibitem[AtE12]{Arendt1}
W.~Arendt and A.~F.~M. ter Elst.
\newblock Sectorial forms and degenerate differential operators.
\newblock {\em J. Operator Theory}, 67(1):33--72, 2012.

\bibitem[BA07]{Benamor07}
A.~Ben~Amor.
\newblock Sobolev-{O}rlicz inequalities, ultracontractivity and spectra of time
  changed {D}irichlet forms.
\newblock {\em Math. Z.}, 255(3):627--647, 2007.

\bibitem[BBST19]{BBST}
H.~BelHadjAli, A.~BenAmor, C.~Seifert, and A.~Thabet.
\newblock On the construction and convergence of traces of forms.
\newblock {\em Journal of Functional Analysis}, 277(5):1334 -- 1361, 2019.

\bibitem[BtE15]{Behrndt-Elst}
J.~Behrndt and A.~F.~M. ter Elst.
\newblock Dirichlet-to-{N}eumann maps on bounded {L}ipschitz domains.
\newblock {\em J. Differential Equations}, 259(11):5903--5926, 2015.

\bibitem[Dan14]{Daners}
D.~Daners.
\newblock Non-positivity of the semigroup generated by the
  {D}irichlet-to-{N}eumann operator.
\newblock {\em Positivity}, 18(2):235--256, 2014.

\bibitem[Dav89]{Davies-book}
E.~B. Davies.
\newblock {\em Heat kernels and spectral theory}, volume~92 of {\em Cambridge
  Tracts in Mathematics}.
\newblock Cambridge University Press, Cambridge, 1989.

\bibitem[FMM98]{Fabes}
E.~Fabes, O.~Mendez, and M.~Mitrea.
\newblock Boundary layers on {S}obolev-{B}esov spaces and {P}oisson's equation
  for the {L}aplacian in {L}ipschitz domains.
\newblock {\em J. Funct. Anal.}, 159(2):323--368, 1998.

\bibitem[JK95]{Jerison-Kenig}
D.~Jerison and C.~E. Kenig.
\newblock The inhomogeneous {D}irichlet problem in {L}ipschitz domains.
\newblock {\em J. Funct. Anal.}, 130(1):161--219, 1995.

\bibitem[JW84]{Jonsson-Wallin}
A.~Jonsson and H.~Wallin.
\newblock Function spaces on subsets of {${\bf R}^n$}.
\newblock {\em Math. Rep.}, 2(1):xiv+221, 1984.

\bibitem[Pos16]{Post}
O.~Post.
\newblock Boundary pairs associated with quadratic forms.
\newblock {\em Math. Nachr.}, 289(8-9):1052--1099, 2016.

\bibitem[Ste70]{Stein}
E.~M. Stein.
\newblock {\em Singular integrals and differentiability properties of
  functions}.
\newblock Princeton Mathematical Series, No. 30. Princeton University Press,
  Princeton, N.J., 1970.

\bibitem[Tar07]{Tartar}
L.~Tartar.
\newblock {\em An introduction to {S}obolev spaces and interpolation spaces},
  volume~3 of {\em Lecture Notes of the Unione Matematica Italiana}.
\newblock Springer, Berlin; UMI, Bologna, 2007.

\end{thebibliography}

\end{document}